\documentclass{article}[a4paper]
\usepackage{amsmath,amssymb,amsthm}
\usepackage{enumitem}

\newtheorem{theorem}{Theorem}[section]	
\newtheorem{lemma}[theorem]{Lemma}
\newtheorem{corollary}[theorem]{Corollary}
\newtheorem{proposition}[theorem]{Proposition}

\theoremstyle{definition}

\newtheorem{remark}[theorem]{Remark}
\newtheorem{example}[theorem]{Example}

\newcommand{\imp}{\mathbin{\to}}

\begin{document}
	\title{Terms that define nuclei on residuated lattices: a case study of BL-algebras}
	
	\author{Sebasti\'an Buss, Diego Casta\~no, Jos\'e Patricio D\'iaz Varela}
		
	\maketitle
	
	\begin{abstract}
		A nucleus $\gamma$ on a (bounded commutative integral) residuated lattice $\mathbf{A}$ is a closure operator that satisfies the inequality $\gamma(a) \cdot \gamma(b) \leq \gamma(a \cdot b)$ for all $a,b \in A$. In this article, among several results, a description of an arbitrary nucleus on a  residuated lattice is given. Special attention is given to terms that define a nucleus on every structure of a variety, as a means of generalizing the double negation operation. Some general results about these terms are presented, together with examples. The main result of this article consists of the description of all terms of this kind for every given subvariety of BL-algebras. We exhibit interesting nontrivial examples.
	\end{abstract}
	
	\section*{Introduction}
	
	Residuated lattices provide equivalent algebraic semantics of extensions of what is known in logic as Full Lambek Calculus, $\mathbf{FL}$ for short (see, for example, \cite{GALATOS2007}). In this article, we will work exclusively with varieties of bounded commutative and integral residuated lattices, also known as $\mathbf{FL}_{\mathbf{ew}}$-algebras, where $\mathbf{FL}_{\mathbf{ew}}$ is the extension of $\mathbf{FL}$ obtained by adding two rules known as $\emph{rule of exchange}$ and $\emph{rule of weakening}$. For simplicity, we shall refer to the members of these varieties as residuated lattices. 
	
	Nuclei were originally defined in the context of Brouwerian algebras and quantales~\cite{ROSENTHAL1990,SCHMIDT1977}. The generalized concept of a nucleus on a residuated lattice was introduced by Galatos in \cite{GALATOS2003}. A \emph{nucleus} on a residuated lattice $\mathbf{A} = \langle A, \land, \lor, \cdot, \imp, \bot, \top \rangle$ is a closure operator (on $\mathbf{A}$) $\gamma$ such that $\gamma(a) \cdot \gamma(b) \leq \gamma(a \cdot b)$ for all $a,b \in A$. An important property of nuclei is that their image can be endowed with a residuated lattice structure: the \emph{nuclear image} of $\mathbf{A}$ with respect to $\gamma$ is the residuated lattice $\mathbf{A}_\gamma := \langle \gamma(A), \land, \lor_\gamma, \cdot_\gamma, \imp, \gamma(\bot), \top \rangle$ where $a \square_{\gamma} b = \gamma(a \square b)$ for $\square \in \lbrace \cdot, \lor \rbrace$ and all $a,b \in \gamma(A)$. The definition of a nucleus can naturally be extended to a broader class of structures, namely, the class of partially ordered monoids. In this context nuclei have been used to derive representation results for different kinds of algebras with an underlying structure of a partially ordered monoid (see, for example, \cite{MUNDICI1986,GALATOS2005,PRENOSIL2022}). More results about nuclei can be found in \cite{GALATOS2007,HAN2011}. Nuclei on residuated lattices were also studied in the context of generalized Bosbach and Rie\v{c}an states \cite{ZHAO2012,ZHAO2013}.
	
	In this work we focus on the study of nuclei on residuated lattices, with special emphasis on divisible prelinear residuated lattices, also known as BL-algebras. In Section~\ref{sec:general_results} we present some general properties. Among some basic results, a description of an arbitrary nucleus on a residuated lattice is given (Theorem~\ref{nucleilocaldescription}). Our main interest lies in nuclei defined by terms. A well-known example of this is the nucleus given by the term $\neg\neg x := (x \imp \bot) \imp \bot$, usually known in the literature as the double negation. For each residuated lattice $\mathbf{A}$, the interpretation $(\neg\neg x)^\mathbf{A}$ is a nucleus on $\mathbf{A}$. This nucleus has been extensively studied in the context of residuated structures and logic (see, for example, \cite{CASTAÑO2011,CASTAÑO2015, CIGNOLI2004, GALATOS2007,GLIVENKO1928}). Based on a definition given in \cite{GALATOS2007}, a term $t(x)$ in the language of residuated lattices is a \emph{nucleus on a variety} $\mathcal{V}$ of residuated lattices if $t^\mathbf{A}$ is a nucleus on $\mathbf{A}$ for all $\mathbf{A} \in \mathcal{V}$. The terms $x, \neg\neg x$ and $\top$ are nuclei on every variety of residuated lattices and will be called \emph{trivial nuclei}. A natural question that arises is if there are varieties $\mathcal{V}$ and terms $t(x)$ such that $t(x)$ is a nucleus on $\mathcal{V}$ and not equivalent (on $\mathcal{V}$) to a trivial nucleus. A term with such properties will be called $\emph{nontrivial nucleus}$ on $\mathcal{V}$. In Section~\ref{sec:nuclei_terms} we present general results about nuclei given by terms and, among other results, examples of varieties with nontrivial nuclei are given (Example~\ref{nontrivial_nucleus_example}). In the case of varieties of BL-algebras, a detailed description of nuclei given by terms is possible. This is the main subject of Section~\ref{sec:nuclei_on_BL}, where a complete description of all nuclei on an arbitrary variety $\mathcal{V}$ of BL-algebras is given (Theorem~\ref{BL_nuclei}). In some cases, nontrivial nuclei exist. A syntactic description of these terms is given.
	
	\section{Preliminaries}
	
	In this section we provide basic definitions that are going to be used throughout this paper, together with some known results. For notation and concepts of universal algebra used in this article we refer the reader to \cite{BURRIS1981}.
	
	A \emph{residuated lattice} \cite{WARD1990} $\mathbf{A} = \langle A, \land, \lor, \cdot, \imp, \bot, \top \rangle$ is an algebra consisting of four binary operations $\lbrace \land, \lor, \cdot, \imp \rbrace$ and two constants $\lbrace \bot, \top \rbrace$ such that 
	\begin{enumerate}
		\item[$\bullet$] $\langle A, \land, \lor, \bot, \top \rangle$ is a bounded lattice,
		\item[$\bullet$] $\langle A, \cdot, \top \rangle$ is a commutative monoid,
		\item[$\bullet$] $a \cdot b \leq c$ if and only if $a \leq b \imp c$, for all $a,b,c \in A$.
	\end{enumerate}
	
	The class of all residuated lattices is a variety algebras \cite{HART2002} that will be denoted by $\mathcal{RL}$. As usual, we define $\neg a:= a \imp \bot$, $a^0 := \top$ and $a^{n+1} := a^n \cdot a$ for all $a \in A$, $n \in \omega$. In the following result we list some basic algebraic properties of residuated lattices (see, for example, \cite{GALATOS2007}).
	\begin{proposition}\label{residuated_lattices_properties}
		Let $\mathbf{A}$ be a residuated lattice and $a,b,c \in A$. Then:
		\begin{enumerate}[label=\rm{(\roman*)}]
			\item $\top \imp a = a$,
			\item $a \leq b$ if and only if $a \imp b = \top$,
			\item if $a \leq b$, then $a \cdot c \leq b \cdot c$, $b \imp c \leq a \imp c$ and $c \imp a \leq c \imp b$,
			\item $a \cdot b \leq a \land b$,
			\item $b \leq a \imp (a \cdot b) \leq a \imp b$,
			\item $a \imp (b \imp c) = (a \cdot b) \imp c = b \imp (a \imp c)$,
			\item $a \cdot (a \imp b) \leq a \land b$,
			\item $a \lor b \leq (a \imp b) \imp b$,
			\item $((a \imp b) \imp b) \imp b = a \imp b$,
			\item if $B \subseteq A$ and $\bigvee B$ exists, then $a \cdot \bigvee B = \bigvee \lbrace a \cdot b: b \in B \rbrace$,
			\item if $B \subseteq A$ and $\bigvee B$ exists, then $\bigvee B \imp a = \bigwedge \lbrace b \imp a: b \in B \rbrace$.
			\item if $B \subseteq A$, $\bigwedge B$ exists, then $a \imp \bigwedge B = \bigwedge \lbrace a \imp b: b \in B \rbrace$.
		\end{enumerate}
	\end{proposition}
	
	Given a residuated lattice $\mathbf{A}$ and $B$ a subset of $A$ we will use $\langle B \rangle^\mathbf{A}$ to denote the subuniverse generated by $B$. For simplicity, the subalgebra associated to the generated subuniverse, the generated subalgebra, will also be denoted $\langle B \rangle^\mathbf{A}$. If $B = \lbrace b_1, \dots,b_n \rbrace$ is finite, then we will use $\langle b_1, \dots,b_m \rangle^\mathbf{A}$ to denote the subuniverse generated by $B$.
	
	A nonempty subset $F$ of $A$ is an \emph{implicative filter on $\mathbf{A}$} if, for all $a,b \in A$, $a \leq b$ and $a \in F$ implies $b \in F$, and $a, b \in F$ implies $a \cdot b \in F$. Equivalently, a subset $F$ of $A$ is an implicative filter if and only if $\top \in F$ and, for all $a,b \in A$, $a,a \imp b \in F$ implies $b \in F$. It is well-known that the lattice of filters on $\mathbf{A}$ is isomorphic to the lattice of congurences on $\mathbf{A}$ (see \cite{GALATOS2007}). 
	
	A residuated lattice $\mathbf{A} = \langle A, \land, \lor, \cdot, \imp, \bot, \top \rangle$ is a \emph{BL-algebra} \cite{HAJEK1998} if it satisfies the two following identities:
	\begin{equation*}
		\begin{aligned}
			x \land y & \approx x \cdot (x \imp y)  & \text{(divisibility)}, \\
			(x \imp y) & \lor (y \imp x) \approx \top  & \text{(prelinearity)}. 
		\end{aligned}
	\end{equation*}
	
	The variety of BL-algebras will be denoted by $\mathcal{BL}$. The finitely subdirectly irreducible members of $\mathcal{BL}$ are none other than the totally ordered BL-algebras, \emph{BL-chains} for short. Given a variety $\mathcal{V}$ of BL-algebras, let $\mathcal{V}_{\textnormal{fsi}}$ denote the class of all BL-chains in $\mathcal{V}$. Since every subdirectly irreducible member of $\mathcal{V}$ is also a member of $\mathcal{V}_{\textnormal{fsi}}$, it follows that $\mathcal{V}_{\textnormal{fsi}}$ generates $\mathcal{V}$ as a variety.
	
	For every BL-algebra $\mathbf{A}$ the identity
	\begin{equation*}
		x \lor y \approx ((x \imp y) \imp y) \land ((y \imp x) \imp x) 
	\end{equation*}
	holds. Since $\land$ can be defined using the operations $\cdot$ and $\imp$, $\mathbf{A}$ is completely determined by its $\lbrace \cdot, \imp, \bot, \top \rbrace$-reduct, that is, $\langle A, \cdot, \imp, \bot, \top \rangle$. A residuated lattice $\mathbf{A} = \langle A, \land, \lor, \cdot, \imp, \bot, \top \rangle$ is said to be an \emph{MV-algebra} \cite{HAJEK1998} if it is a BL-algebra such that the identity 
	\begin{equation*}
		x \approx \neg\neg x \ \text{(involution)}
	\end{equation*}
	holds. 
	
	A \emph{hoop} \cite{FERREIRIM1992} is an algebra $\mathbf{A} = \langle A, \cdot, \imp, \top \rangle$ such that $\langle A, \cdot, \top \rangle$ is a commutative monoid and the following identities hold:
	\begin{equation*}
		\begin{aligned}
			x \imp x & \approx \top, \\
			x \cdot(x \imp y) & \approx y \cdot(y \imp x), \\
			x \imp (y \imp x) & \approx (x \cdot y) \imp z.
		\end{aligned}
	\end{equation*}
	
	If we define the relation $x \leq y$ if and only if $x \imp y = \top$, then $\leq$ defines a semilattice order on $\mathbf{A}$, with $x \land y = x \cdot (x \imp y)$ and last element $\top$. Moreover, $a \cdot b \leq c$ if and only if $a \leq b \imp c$ for all $a,b,c \in A$.
	
	When the order of a hoop $\mathbf{A}$ is total we say that $\mathbf{A}$ is a \emph{totally ordered hoop}. A \emph{Wajsberg hoop} is a hoop that satisfies
	\begin{equation*}\label{tanaka}
		(x \imp y) \imp y \approx y \imp (y \imp x).
	\end{equation*}
	
	A \emph{bounded hoop} is an algebra $\mathbf{A} = \langle A, \cdot, \imp, \bot, \top \rangle$ such that $\langle A, \cdot, \imp, \top \rangle$ is a hoop and $\bot \leq a$ for each $a \in A$. If $\mathbf{A} = \langle A, \cdot, \imp, \top \rangle$ is a hoop with first element $\bot$, then we define the bounded hoop $\mathbf{A}^+ := \langle A, \cdot, \imp, \bot, \top \rangle$. A \emph{Wajsberg algebra} is a bounded Wajsberg hoop. A hoop $\mathbf{A}$ is \emph{cancellative} if it satisfies the identity
	\begin{equation*}
		x \imp (x \cdot y) \approx y.
	\end{equation*}
	
	A \emph{basic hoop} \cite{AGLIANO2007} is a hoop which satisfies the identity
	\begin{equation*}
		((x \imp y) \imp z) \imp (((y \imp x) \imp z) \imp z) \approx \top.
	\end{equation*}
	
	Clearly, every totally ordered hoop is basic. We continue with some examples that are going to be used in this paper. 
	
	\begin{example}
		
		Consider $S_n = \lbrace a^i: i \in \lbrace 0, \dots,n \rbrace \rbrace$ with the order
		\begin{equation*}
			\bot = a^n < \dots < a < a^0 = \top
		\end{equation*}
		and the operations $a^i \cdot a^j := a^{\min\lbrace i+j,n \rbrace}$ and $a^i \imp a^j := a^{\max\lbrace j-i,0 \rbrace}$ for all $i,j \in \lbrace 0, \dots,n \rbrace$. Then, $\mathbf{S}_n := \langle S_n, \cdot, \imp, \top \rangle$ is a totally ordered Wajsberg hoop with first element $a^n$. Therefore, $\mathbf{S}_n^+ = \langle S_n, \cdot, \imp, a^n, \top \rangle$ is a bounded totally ordered Wajsberg hoop. 
		
		Next consider the set $S_\omega = \lbrace a^i: i \in \omega \rbrace$ with the total order $a^{i} \leq a^j$ iff $i \geq j$ and the operations $a^i \cdot a^j := a^{i+j} \quad \text{and} \quad a^i \imp a^j := a^{\max\lbrace j-i,0 \rbrace}$. The structure $\mathbf{S}_\omega := \langle S_\omega, \cdot, \imp, \top \rangle$ is a cancelative totally ordered Wajsberg hoop. 
		
		Finally, consider $A = \lbrace (0,k): k \in \omega \rbrace \cup \lbrace (1,-k): k \in \omega \rbrace$. If we define 
		\begin{gather*} 
			\bot := (0,0), 
			\quad
			\top := (1,0), 
		\end{gather*}
		\begin{equation*}
			a \cdot b := 
			\begin{cases}
				(0,0) & \text{if $a = (0,j)$, $b = (0,k)$, $j,k \geq 0$;} \\
				(1,-(j+k)) & \text{if $a = (1,-j)$, $b = (1,-k)$, $j,k \geq 0$;} \\
				(0,0) & \text{if $a = (0,j)$, $b = (1,-k)$, $j,k \geq 0$, $k > j$;} \\
				(0,j-k) & \text{if $a = (0,j)$, $b = (1,-k)$, $j,k \geq 0$, $k \leq j$},
			\end{cases}
		\end{equation*}
		\begin{equation*}
			a \imp b := 
			\begin{cases}
				(1,0) & \text{if $a \leq b$;} \\
				(1,k-j) & \text{if $a = (0,j)$, $b = (0,k)$, $j,k \geq 0$, $j > k$;} \\
				(0,j+k) & \text{if $a = (1,-j)$, $b = (0,k)$, $j,k \geq 0$;} \\
				(1,j-k) & \text{if $a = (1,-j)$, $b = (1,-k)$, $j,k \geq 0$, $j < k$}
			\end{cases}
		\end{equation*}
		for all $a,b \in A$, then the algebra $\mathbf{S}_1^\omega := \langle A, \cdot, \imp, \top \rangle$ is a totally ordered Wajsberg hoop. Moreover, $\bot$ is the first element of this structure and $(\mathbf{S}_1^\omega)^+$ is a bounded hoop.
		
	\end{example}
	
	Bounded basic hoops are term equivalent to BL-algebras \cite[Theorem~1.6]{AGLIANO2007}, same as Wajsberg algebras to MV-algebras.
	\begin{theorem}\label{BL-basic_term_equivalence}
		If $\langle A, \land, \lor, \cdot, \imp, \bot, \top \rangle$ is a BL-algebra, then $\langle A, \cdot, \imp, \bot, \top \rangle$ is a bounded basic hoop. Moreover, if $\langle B, \cdot, \imp, \bot, \top \rangle$ is a bounded basic hoop, then $\langle B, \land, \lor, \cdot, \imp, \bot, \top \rangle$ is a BL-algebra, where
		\begin{equation*}
			x \land y := x\cdot(x \imp y) \quad \text{and} \quad x \lor y := ((x \imp y) \imp y) \land ((y \imp x) \imp x).
		\end{equation*}
		
		This correspondence establishes a term equivalence between BL-algebras and bounded basic hoops. The same correspondence also establishes a term equivalence between MV-algebras and Wajsberg algebras.
	\end{theorem}
	
	In view of the last theorem, we are not going to distinguish between bounded basic hoops and BL-algebras, same as between MV-algebras and Wajsberg algebras. The respective definitions will be considered equivalent and, if necessary, the correspondence given by the previous theorem will be used.
	
	Let $\langle I, \leq \rangle$ be a total order. For each $i \in I$ let $\mathbf{A}_i = \langle A_i, \cdot_i, \imp_i, \top \rangle$ be a hoop such that $A_i \cap A_j = \lbrace \top \rbrace$ for all $j \in I$, $i \neq j$. The \emph{ordinal sum} is the hoop $\bigoplus_{i \in I} \mathbf{A}_i := \langle \cup_{i \in I} A_i, \cdot, \imp, \top \rangle$ where the operations are defined as follows:
	\begin{equation*}
		x \cdot y :=
		\begin{cases}
			x \cdot_i y & \text{if $x,y \in A_i$}; \\
			x & \text{if $x \in A_i \backslash \lbrace \top \rbrace$, $y \in A_j$, $i < j$}; \\
			y & \text{if $y \in A_i \backslash \lbrace \top \rbrace$, $x \in A_j$, $i < j$}.
		\end{cases}
	\end{equation*}
	and
	\begin{equation*}
		x \imp y :=
		\begin{cases}
			\top & \text{if $x \in A_i \backslash \lbrace \top \rbrace$, $y \in A_j$, $i < j$}; \\
			x \imp_i y & \text{if $x,y \in A_i$}; \\
			y & \text{if $y \in A_i$, $x \in A_j$, $i < j$}.
		\end{cases}
	\end{equation*}
	
	Each $\mathbf{A}_i$ will be called \emph{component} of the ordinal sum. If $I$ has first element $0$ and $\mathbf{A}_0$ is a hoop with first element $\bot$, then $\bot$ is the first element of $\bigoplus_{i \in I} \mathbf{A}_i$. In this case, $(\bigoplus_{i \in I} \mathbf{A}_i)^+$ is a bounded hoop. The subalgebras of an ordinal sum can be described easily \cite[Proposition~3.1]{AGLIANO2003}.
	\begin{proposition}\label{subhoops_of_ordinal_sum}
		Let $\langle I, \leq \rangle$ be a total order and consider $\lbrace \mathbf{A}_i: i \in I \rbrace$ a set of hoops such that $A_i \cap A_j = \lbrace \top \rbrace$ for all $i \neq j$. The subalgebras of $\bigoplus_{i \in I} \mathbf{A}_i$ are exactly the structures of the form $\bigoplus_{j \in J} \mathbf{B}_j$, where $\langle J, \leq \rangle$ is a subposet of $\langle I, \leq\rangle$ and $\mathbf{B}_j$ is a subalgebra of $\mathbf{A}_j$ for all $j \in J$. If $I$ has first element $0$ and $\mathbf{A}_0$ has first element $\bot$, then the subalgebras of $(\bigoplus_{i \in I} \mathbf{A}_i)^+$ are exactly the structures of the form $(\bigoplus_{j \in J} \mathbf{B}_j)^+$ where $\langle J, \leq \rangle$ is a subposet of $\langle I, \leq\rangle$, $\mathbf{B}_j$ is a subalgebra of $\mathbf{A}_j$ for all $j \in J$, $0 \in J$ and $\bot \in B_0$.
	\end{proposition}
	
	Ordinal sums play a key role in the description of totally ordered hoops and, in particular, BL-chains. The following decomposition is a well-known result in the theory of hoops. A proof of this fact may be found in \cite{BUSANICHE2005}.
	\begin{theorem}\label{BLchains_decomposition}
		For each nontrivial BL-chain $\mathbf{A}$ there exists a total order $\langle I, \leq~\rangle$ with first element $0$ and a set of nontrivial totally ordered Wajsberg hoops $\lbrace \mathbf{A}_i: i \in I \rbrace$ such that $A_i \cap A_j = \lbrace \top \rbrace$ for all $i \neq j$, $\mathbf{A}_0$ has first element $\bot^\mathbf{A}$ and $\mathbf{A} = (\bigoplus_{i \in I} \mathbf{A}_i)^+$. Conversely, if $\langle I, \leq \rangle$ is a total order with first element $0$ and $\lbrace \mathbf{A}_i: i \in I \rbrace$ is a set of nontrivial totally ordered Wajsberg hoops such that $A_i \cap A_j = \lbrace \top \rbrace$ for all $i \neq j$ and $\mathbf{A}_0$ has first element, then $(\bigoplus_{i \in I} \mathbf{A}_i)^+$ is a nontrivial BL-chain.
	\end{theorem}	
	
	Given $\lbrace \mathbf{A}_i: i \in I \rbrace$ a set of hoops, it is always possible to find another set of hoops $\lbrace \mathbf{B}_i: i \in I \rbrace$ such that $\mathbf{A}_i$ is isomorphic to $\mathbf{B}_i$ and $B_i \cap B_j = \lbrace \top \rbrace$ for all $i,j \in I$, $i \neq j$. Therefore, for simplicity, we can always safely assume that the condition $A_i \cap A_j = \lbrace \top \rbrace$ for all $i \neq j$ is true. As an example, the structures $(\mathbf{S}_n \oplus \mathbf{S}_n)^+$ and $(\mathbf{S}_n \oplus \mathbf{S}_\omega)^+$ are going to be considered BL-chains.
	
	Let $\mathbf{A}$ be a totally ordered Wajsberg hoop and $a,b,c \in A$ such that $c < a \cdot b$. Then $a \cdot b \nleq c$ and $a \nleq b \imp c$. Therefore $b \imp c \leq a$ and
	\begin{equation*}
		\begin{aligned}
			a & = (a \imp (b \imp c)) \imp (b \imp c) \\
			& = ((a \cdot b) \imp c) \imp (b \imp c) \\
			& = b \imp (((a \cdot b) \imp c) \imp c) \\
			& = b \imp (a \cdot b).
		\end{aligned}
	\end{equation*}
	
	The following result is an immediate consequence of the previous observation and a well-known property of totally ordered Wajsberg hoops (see, for example, \cite[Lemma~4.5, Proposition~4.6]{FERREIRIM1992}). We say that a hoop $\mathbf{A}$ is \emph{semi-cancellative} if $c < a \cdot b$ implies $a \imp (a \cdot b) = b$, for all $a,b,c \in A$.
	\begin{proposition}\label{to_wajsberg_is_semi_cancellative}
		Every totally-ordered Wajsberg hoop is semi-cancellative. As a consequence, a nontrivial totally-ordered Wajsberg hoop is either bounded or cancellative. 
	\end{proposition}
	
	Now consider $\mathbf{A} = (\bigoplus_{i \in I} \mathbf{A}_i)^+$, where $\langle I, \leq \rangle$ is a total order with first element $0$ and $\lbrace \mathbf{A}_i: i \in I \rbrace$ is a set of nontrivial totally ordered Wajsberg hoops such that $A_i \cap A_j = \lbrace \top \rbrace$ for all $i \neq j$ and $\mathbf{A}_0$ has first element $\bot$. By semi-cancellativity, if $a \neq \top$, then $a^2 = a$ if and only if $a$ is the first element of some $\mathbf{A}_i$. Notice that, as a consequence of the definition of ordinal sums,
	\begin{equation}\label{BLchains_negation}
		\neg\neg a =
		\begin{cases}
			a & \text{if $a \in A_0$}, \\
			\top & \text{if $a \in A_i$, $i > 0$}.
		\end{cases}
	\end{equation}
	
	Considering these remarks and Proposition~\ref{to_wajsberg_is_semi_cancellative} we obtain the following result.
	\begin{proposition}\label{BL_chains_generated_subalgebras}
		Let $\langle J, \leq \rangle$ be a nonempty subposet of $\langle I, \leq \rangle$ such that $0 \notin J$. Consider $\langle a_j: j \in J \rangle \in \prod_{j \in J} A_j$ such that, for all $j \in J$, $a_j^2 = a_j$ or $a_j^{n+1} < a_j^n$ for all $n \in \omega$. Then, $\langle \lbrace a_j: j \in J \rbrace \rangle^\mathbf{A} = \bigcup_{j \in \lbrace 0 \rbrace \cup J} B_j$ where
		\begin{equation*}
			B_j = 
			\begin{cases}
				\lbrace \bot, \top \rbrace & \text{if $j = 0$}, \\
				\lbrace a_j, \top \rbrace & \text{if $j > 0$ and $a_j^2 = a_j$}, \\
				\lbrace a_j^n: n \in \omega \rbrace & \text{if $j > 0$ and $a_j^{n+1} < a_j^n$ for all $n \in \omega$}.
			\end{cases}
		\end{equation*}
		
		Therefore, $\langle \lbrace a_j: j \in J \rbrace \rangle^\mathbf{A} \cong (\bigoplus_{j \in \lbrace 0 \rbrace \cup J} \mathbf{B}_j)^+$ where 
		\begin{equation*}
			\mathbf{B}_j = 
			\begin{cases}
				\mathbf{S}_1 & \text{if $j = 0$ or $j > 0$ and $a_j^2 = a_j$}, \\
				\mathbf{S}_\omega & \text{if $j > 0$ and $a_j^{n+1} < a_j^n$ for all $n \in \omega$}.
			\end{cases}
		\end{equation*}
	\end{proposition}
	
	\section{General results about nuclei on residuated lattices}\label{sec:general_results}
	
	The goal of this section is to summarize general properties of nuclei on residuated lattices. Among various results a description of an arbitrary nucleus is given (Theorem~\ref{nucleilocaldescription}).
	
	A \emph{closure operator} on a poset $\mathbf{P}$ is a map $\gamma \colon P \to P$ such that 
	\begin{equation*}
		a \leq \gamma(a), \quad a \leq b \Rightarrow \gamma(a) \leq \gamma(b) \quad \text{and} \quad \gamma(a) = \gamma(\gamma(a))
	\end{equation*}
	for all $a,b \in P$. Since $\gamma(a) = \min\lbrace b \in \gamma(P): a \leq b \rbrace$ for all $a \in P$, closure operators are uniquely determined by its image $\gamma(P) = \lbrace a \in P: \gamma(a) = a \rbrace$. 
	
	A $\emph{nucleus}$ on a residuated lattice $\mathbf{A}$ is a closure operator $\gamma$ on the lattice reduct of $\mathbf{A}$ satisfying $\gamma(a) \cdot \gamma(b) \leq \gamma(a \cdot b)$ for all $a,b \in A$. It can be shown that an arbitrary map $\gamma$ on a residuated lattice $\mathbf{A}$ is a nucleus if and only if
	\begin{equation}\label{nuclei_equation}
		a \imp \gamma(b) = \gamma(a) \imp \gamma(b)
	\end{equation}
	for all $a,b \in A$ \cite{ROSENTHAL1990}. If $\mathbf{A} = \langle A, \land, \lor, \cdot, \imp, \bot, \top \rangle$ is a residuated lattice and $\gamma$ a nucleus on $\mathbf{A}$, then the image $\gamma(A)$ can be endowed with a residuated lattice structure \cite{GALATOS2005} by defining $\mathbf{A}_\gamma := \langle \gamma(A), \land, \lor_\gamma, \cdot_\gamma, \imp, \gamma(\bot), \top \rangle$ where
	\begin{equation*}
		\gamma(a) \lor_\gamma \gamma(b) := \gamma(\gamma(a) \lor \gamma(b)) \quad \text{and} \quad \gamma(a) \cdot_\gamma \gamma(b) := \gamma(\gamma(a) \cdot \gamma(b))
	\end{equation*}
	for all $a,b \in A$. The algebra $\mathbf{A}_\gamma$ is known as the \emph{nuclear image of $\mathbf{A}$ relative to $\gamma$}. We say that $a \in A$ is $\gamma$-\emph{dense} if $\gamma(a) = \top$. It can be easily show that the set of all $\gamma$-dense elements, which will be denoted $D_\gamma(A)$, is an implicative filter on $\mathbf{A}$, that is, $\top \in  D_\gamma(A)$ and if $a, a \imp b \in D_\gamma(A)$, then $b \in D_\gamma(A)$. 
	\begin{example}\label{nuclei_examples}
		Let $a$ be a fixed element of $\mathbf{A}$. If we define  
		\begin{equation*}
			\lor_a x := x \lor a \quad \text{and} \quad \neg_a \neg_a x := (x \imp a) \imp a,
		\end{equation*}
		then the functions $\lor_a$ and $\neg_a\neg_a$ are nuclei on $\mathbf{A}$: by Proposition~\ref{residuated_lattices_properties}, 
		\begin{equation*}
			\begin{aligned}
				(x \lor a) \imp (y \lor a) & = [x \imp (y \lor a)] \land [a \imp (y \lor a)] \\
				& = x \imp (y \lor a)
			\end{aligned}
		\end{equation*}
		and 
		\begin{equation*}
			\begin{aligned}
				\lbrack (x \imp a) \imp a \rbrack \imp \lbrack  (y \imp a) \imp a \rbrack & = (y \imp a) \imp \lbrack  ((x \imp a) \imp a) \imp a \rbrack \\
				& = (y \imp a) \imp (x \imp a) \\
				& = x \imp ((y \imp a) \imp a).
			\end{aligned}
		\end{equation*}
		
		Moreover, $\lor_a(A) = \lbrace b \in A: a \leq b \rbrace$ and $\neg_a\neg_a(A) = \lbrace b \imp a: b \in A \rbrace$. The function $\neg_a x := x \imp a$ is known as the \emph{relative negation} \cite{ZHAO2012} (to the element $a$) and generalizes the classical negation $\neg x = x \imp \bot$.
	\end{example}
	
	We continue with some basic algebraic properties of nuclei on residuated lattices. Some of them are well-known results for closure operators on partially ordered sets (see, for example, \cite{PRIESTLEY2002}). Let $\mathbf{A}$ be a residuated lattice and $\gamma$ a nucleus on $\mathbf{A}$. 
	\begin{lemma}\label{nuclei_properties1}
		For all $a, b \in A$ the following hold:
		\begin{enumerate}[label=\rm{(\roman*)}]
			\item if $B \subseteq A$ and $\bigwedge \gamma(B)$ exists, then $\gamma(\bigwedge \gamma(B)) = \bigwedge \gamma(B)$,
			\item $\gamma(a \lor b) = \gamma(\gamma(a) \lor b) = \gamma(\gamma(a) \lor \gamma(b))$,
			\item $\gamma(\bot) = \min \gamma(A)$ and $\gamma(\top) = \top$,
			\item $\gamma(ab) = \gamma(\gamma(a)b) = \gamma(\gamma(a)\gamma(b))$,
			\item $\gamma(a \imp b) \leq \gamma(a) \imp \gamma(b) \leq a \imp \gamma(b)$,
			\item $\gamma(a \imp \gamma(b)) = a \imp \gamma(b)$,
			\item if $b \in A$ and $\gamma(b) \leq \gamma(a)$, then $a \lor \gamma(b) \leq \gamma(a) \leq (a \imp \gamma(b)) \imp \gamma(b)$,
			\item $a \lor \gamma(\bot) \leq \gamma(a) \leq (a \imp \gamma(\bot)) \imp \gamma(\bot)$.
		\end{enumerate}
	\end{lemma}
	\begin{proof}
		We restrict ourselves to the proofs of items (iv) to (vii). To prove (iv) notice that $a \cdot b \leq \gamma(a) \cdot b \leq \gamma(a) \cdot \gamma(b) \leq \gamma(a \cdot b)$. Since $\gamma(a \cdot b) = \gamma(\gamma(a \cdot b))$ and $\gamma$ is monotone we obtain the desired result. By Proposition~\ref{residuated_lattices_properties}, $a \cdot (a \imp b) \leq b$. Therefore $\gamma(a)\gamma(a \imp b) \leq \gamma(a \cdot (a \imp b)) \leq \gamma(b)$, or equivalently, $\gamma(a \imp b) \leq \gamma(a) \imp \gamma(b)$. Also notice that $\gamma(a) \imp b \leq a \imp b$, because $a \leq \gamma(a)$. This ends the proof item (v). The inequality $a \imp \gamma(b) \leq \gamma(a \imp \gamma(b))$ is immediate. Moreover, notice that $a \leq \gamma(a) \leq (\gamma(a) \imp \gamma(b)) \imp \gamma(b)$. As a consequence of item (v),
		\begin{equation*}
			\gamma(a \imp \gamma(b)) \leq \gamma(a) \imp \gamma(\gamma(b)) = \gamma(a) \imp \gamma(b) \leq a \imp \gamma(b).
		\end{equation*}
		
		This concludes the proof of item (vi). To prove item (vii) first notice that $a \leq (a \imp \gamma(b)) \imp \gamma(b)$. Therefore, using item (vi) and the monotonicity of $\gamma$, 
		\begin{equation*}
			a \lor \gamma(b) \leq \gamma(a) \leq \gamma((a \imp \gamma(b)) \imp \gamma(b)) = (a \imp \gamma(b)) \imp \gamma(b).
		\end{equation*} 
	\end{proof}
	
	As a direct consequence of Lemma~\ref{nuclei_properties1} and previous examples we obtain the following characterization of nuclei on MV-algebras \cite[Corollary~2.10]{HAN2011}. Recall that, if $\mathbf{A}$ is an MV-algebra, then $a \lor b = (a \imp b) \imp b$ for all $a,b \in A$.
	\begin{corollary}\label{MV_nuclei}
		If $\mathbf{A}$ is an MV-algebra, then $\gamma$ is a nucleus on $\mathbf{A}$ if and only if $\gamma(a) = a \lor \gamma(\bot)$ for all $a \in A$.
	\end{corollary}
	
	To end this section we provide a description of all nuclei on an arbitrary residuated lattice $\mathbf{A}$. We start with a characterization of the elements in the image of a nucleus \cite[Proposition~2.8]{HAN2011}. 
	\begin{lemma}\label{nuclei_image_characterization}
		If $\mathbf{A}$ is a residuated lattice, $\gamma$ is a nucleus on $\mathbf{A}$ and $a \in A$, then $\gamma(a) = a$ if and only if $\gamma(b) \leq \neg_a\neg_a b$ for all $b \in A$.
	\end{lemma} 
	
	The previous lemma will be used frequently in the following results and sections. Before presenting the main result of this section we need some preliminaries.
	\begin{proposition}\label{nuclei_properties2} Let $\mathbf{A}$ be a residuated lattice.
		\begin{enumerate}[label=\rm{(\roman*)}]
			\item If $\gamma$ is a nucleus on $\mathbf{A}$, then $\gamma(a) = \min \lbrace \neg_b\neg_b a: b \in \gamma(A) \rbrace$ for all $a \in A$.
			\item If $\lbrace \gamma_i: i \in I \rbrace$ is a set of nuclei on $\mathbf{A}$ such that $\gamma(a) := \bigwedge \lbrace \gamma_i(a): i \in I \rbrace$ exists for all $a \in A$, then $\gamma$ is a nucleus on $\mathbf{A}$.
		\end{enumerate}
	\end{proposition}
	\begin{proof} 
		Suppose $\gamma$ is a nucleus on $\mathbf{A}$ and let $a \in A$. By Lemma~\ref{nuclei_image_characterization} we have $\gamma(a) \leq \neg_b \neg_b a$ for all $b \in \gamma(A)$. Moreover, $\gamma(a) \in \gamma(A)$ and $\neg_{\gamma(a)} \neg_{\gamma(a)} a = \gamma(a)$. Now consider the hypothesis in (ii) and let $a,b \in A$. Since $a \leq \gamma(a)$, it is clear that $\gamma(a) \imp \gamma(b) \leq a \imp \gamma(b)$. Moreover,
		\begin{equation*}
			\begin{aligned}
				a \imp \textstyle \bigwedge \lbrace \gamma_i(b) : i \in I \rbrace & = \textstyle \bigwedge \lbrace a \imp \gamma_i(b): i \in I \rbrace \\
				& = \textstyle \bigwedge \lbrace \gamma_i(a) \imp \gamma_i(b): i \in I \rbrace \\
				& \leq \textstyle \bigwedge \lbrace \gamma(a) \imp \gamma_i(b): i \in I \rbrace \\
				& = \gamma(a) \imp \textstyle \bigwedge \lbrace \gamma_i(b) : i \in I \rbrace. 
			\end{aligned}
		\end{equation*} 
	\end{proof}
	
	The previous result greatly improves a result that was already proven in the context of nuclei on frames, also known as complete Heyting algebras or complete Brouwerian algebras (see \cite[Lemma 7]{SIMMONS1978}). 
	
	As an immediate consequence of Example~\ref{nuclei_examples} and Proposition~\ref{nuclei_properties2} be obtain the following description for an arbitrary nucleus on a residuated lattice $\mathbf{A}$.
	\begin{theorem}\label{nucleilocaldescription}
		Let $\gamma$ be a function on $\mathbf{A}$. The following are equivalent:
		\begin{enumerate}[label=\rm{(\roman*)}]
			\item $\gamma$ is a nucleus on $\mathbf{A}$;
			\item for some $B \subseteq A$, $\gamma(x) = \min \lbrace \neg_b \neg_b x: b \in B \rbrace$;
			\item for some $B \subseteq A$, $\gamma(x) = \bigwedge \lbrace \neg_b \neg_b x: b \in B \rbrace$. 
		\end{enumerate}
	\end{theorem}
	
	The previous theorem allows us to easily obtain all nuclei on a residuated lattice $\mathbf{A}$ whose lattice reduct is complete (for example, if $A$ is finite).
	
	\section{Nuclei defined by terms}\label{sec:nuclei_terms}
	
	As we previously mentioned in the introduction, our main interest lies in nuclei defined by terms. The identity function $x^\mathbf{A}$, the double negation $(\neg\neg x)^\mathbf{A}$ and the constant function $\top^\mathbf{A}$ all define nuclei on every residuated lattice $\mathbf{A}$ (see Example~\ref{nuclei_examples}). The terms $x$ and $\top$ are uninteresting in the sense that the nuclear image of a residuated lattice relative to these nuclei is either the same structure or trivial. On the other hand, the term $\neg\neg x$ is a well-studied nucleus in the context of residuated structures and logic (for references see the introduction).
	
	We can ask ourselves if there are other terms different from $x$, $\neg\neg x$ and $\top$ that define nuclei and if known results for $\neg\neg x$ also hold for these terms. Indeed, as we are going to show in this section and the following, there are terms $t(x)$ and varieties of residuated lattices $\mathcal{V}$ such that $t^\mathbf{A}$ is a nucleus for every $\mathbf{A} \in \mathcal{V}$ and $\mathcal{V} \nvDash t(x) \approx s(x)$ for all $s(x) \in \lbrace x, \neg\neg x, \top \rbrace$. Furthermore, in the context of varieties of BL-algebras, these terms and $\neg\neg x$ share some similar properties.
	
	Since terms can be interpreted on arbitrary structures, the notion of a nucleus defined by a term arises naturally. Following ideas in \cite[Section~8.6]{GALATOS2007}, a term $t(x)$ in the language of residuated lattices is said to be a \emph{nucleus on a variety} $\mathcal{V} \subseteq \mathcal{RL}$ if the interpretation $t^\mathbf{A}$ is a nucleus on $\mathbf{A}$ for all $\mathbf{A} \in \mathcal{V}$. Using the equality given in \eqref{nuclei_equation} we obtain the following equivalent definition.
	\begin{proposition}
		Let $\mathcal{V}$ be a variety of residuated lattices. Then $t(x)$ is a nucleus on $\mathcal{V}$ if and only if $\mathcal{V} \vDash t(x) \imp t(y) \approx x \imp t(y)$.
	\end{proposition}
	
	By the previous proposition, for every term $t(x)$ there exists a variety $\mathcal{V}$ such that $t(x)$ is a nucleus on $\mathcal{V}$. It suffices to consider the variety $\mathcal{V}_t$ defined, relative to $\mathcal{RL}$, by the identity $t(x) \imp t(y) \approx x \imp t(y)$.
	
	Notice that $\mathcal{V} \vDash t(\bot) \approx \bot$ or $\mathcal{V} \vDash t(\bot) \approx \top$	for each variety of residuated lattices $\mathcal{V}$ and every unary term $t(x)$, because $\lbrace \bot^\mathbf{A}, \top^\mathbf{A} \rbrace$ is a subuniverse of $\mathbf{A}$ for all $\mathbf{A} \in \mathcal{V}$. Taking into account this remark we obtain the next result.
	\begin{proposition}\label{nuclei_term_interval}
		Let $\mathcal{V} \subseteq \mathcal{RL}$ be a variety and $t(x)$ a nucleus on $\mathcal{V}$.
		\begin{enumerate}[label=\rm{(\roman*)}]
			\item If $\mathcal{V} \vDash t(\bot) \approx \top$, then $\mathcal{V} \vDash t(x) \approx \top$.
			\item If $\mathcal{V} \vDash t(\bot) \approx \bot$, then $\mathcal{V} \vDash (x \imp t(x)) \land (t(x) \imp \neg\neg x) \approx \top$.
		\end{enumerate}
	\end{proposition}
	\begin{proof}
		Let $\mathbf{A} := \mathbf{F}_{\mathcal{V}}(x)$ be the free one-generated algebra in $\mathcal{V}$ and suppose that $\mathcal{V} \vDash t(\bot) \approx \top$. Given that $\mathbf{A} \in \mathcal{V}$ and $t(x)$ is a nucleus on $\mathcal{V}$, $\top^\mathbf{A} = t^\mathbf{A}(\bot^\mathbf{A}) \leq t^\mathbf{A}(a)$ for all $a \in A$. Therefore $\mathbf{A} \vDash t(x) \approx \top$. Moreover, $\mathcal{V} \vDash t(x) \approx \top$. Now suppose that $\mathcal{V} \vDash t(\bot) \approx \bot$. Using Lemma~\ref{nuclei_properties1} we obtain that $\mathbf{A} \vDash (x \imp t(x)) \land (t(x) \imp \neg\neg x) \approx \top$. Thus, $\mathcal{V}$ satisfies the same identity.
	\end{proof}
	
	The terms $x$, $\neg\neg x$ and $\top$ are classic examples of nuclei on any variety of residuated lattices. A term $t(x)$ will be called a \emph{trivial nucleus} on a variety $\mathcal{V}$ if $\mathcal{V} \vDash t(x) \approx s(x)$ for some $s(x) \in \lbrace x, \neg\neg x, \top \rbrace$. On the other hand, a term $t(x)$ will be a \emph{nontrivial nucleus} on a variety $\mathcal{V}$ if $t(x)$ is a nucleus on $\mathcal{V}$ such that $\mathcal{V} \nvDash t(x) \approx s(x)$ for all $s(x) \in \lbrace x, \neg\neg x, \top \rbrace$. Notice that $\mathcal{V} \nvDash t(x) \approx \top$ is equivalent to $\mathcal{V} \vDash t(\bot) \approx \bot$. Moreover, $t(x)$ is a nontrivial nucleus on $\mathcal{V}$ if and only if $t(x)$ is a nucleus on $\mathcal{V}$, $\mathcal{V} \vDash t(\bot) \approx \bot$ and there exist $\mathbf{A}$, $\mathbf{B} \in \mathcal{V}$ and $a \in A$, $b \in B$ such that $a < t^\mathbf{A}(a)$ and $t^\mathbf{B}(b) < \neg\neg b$.
	
	A nucleus $t(x)$ on a variety $\mathcal{V}$ can be seen as an element of the free one-generated algebra on $\mathcal{V}$, $\mathbf{F}_\mathcal{V}(x)$. By Proposition~\ref{nuclei_term_interval}, it is equal to the last element or it lies in the interval $[x, \neg\neg x]$. A good description of the interval $[x, \neg\neg x]$ in $\mathbf{F}_\mathcal{V}(x)$ may be a useful tool for finding all nuclei over a variety $\mathcal{V}$. Unfortunately, in most cases, this description can be a nontrivial task. We follow up with some examples. 
	
	Let $\mathcal{IRL}$ be the subvariety of $\mathcal{RL}$ defined by the identity $x \approx \neg\neg x$. Algebras of $\mathcal{IRL}$ are known as \emph{involutive residuated lattices}. A very well-known and studied example of a subvariety of $\mathcal{IRL}$ is the variety of MV-algebras, that is, the variety of involutive BL-algebras. The following result is an immediate consequence of the definition of $\mathcal{IRL}$.
	\begin{proposition}\label{nuclei_involutive_varieties}
		Let $\mathcal{V}$ be a subvariety of $\mathcal{IRL}$. Then $t(x)$ is a nucleus on $\mathcal{V}$ if and only if $\mathcal{V} \vDash t(x) \approx s(x)$ for some $s(x) \in \lbrace x, \top \rbrace$ (equivalently $\lbrace \neg\neg x, \top \rbrace$). 
	\end{proposition}
	
	Let $\mathcal{H}$ denote the subvariety of $\mathcal{RL}$ defined by the identity $x^2 \approx x$ (idempotence). Elements of $\mathcal{H}$ are known as \emph{Heyting algebras}. In this case there exists a well-known description of $\mathbf{F}_{\mathcal{H}}(x)$ \cite{NISHIMURA1960} and as shown in this description, the open interval $(x, \neg\neg x)$ is empty, or equivalently, $\neg\neg x$ covers $x$. Therefore, only trivial nuclei exist on $\mathcal{H}$. Moreover, each subvariety of $\mathcal{H}$ satisfies the same property.
	\begin{proposition}
		Let $\mathcal{V}$ be a subvariety of $\mathcal{H}$. Then $t(x)$ is a nucleus on $\mathcal{V}$ if and only if $\mathcal{V} \vDash t(x) \approx s(x)$ for some $s(x) \in \lbrace x, \neg\neg x, \top \rbrace$. 
	\end{proposition}
	\begin{proof}
		Observe that there exists a surjective homomorphism from $\mathbf{F}_\mathcal{H}(x)$ onto $\mathbf{F}_\mathcal{V}(x)$: the partial map $x \mapsto x$ can be extended to a surjective homomorphism $\alpha \colon \mathbf{F}_\mathcal{H}(x) \to \mathbf{F}_\mathcal{V}(x)$, as a consequence of the universal mapping property of free algebras. Moreover, $\alpha(x) = x$ and $\alpha(\neg\neg x) = \neg\neg x$. Furthermore, for each surjective lattice homomorphism $f \colon \mathbf{L}_1 \to \mathbf{L}_2$, if $b$ covers $a$, then $f(a) = f(b)$ or $f(b)$ covers $f(a)$. 
	\end{proof}
	
	An example and a well-known subvariety of $\mathcal{H}$ is the variety of \emph{G\"odel algebras} denoted by $\mathcal{G}$ and defined (relative to $\mathcal{H}$) by the identity $(x \imp y) \lor (y \imp x) \approx \top$.
	
	Until now we gave examples of varieties that only have trivial nuclei. There exist several varieties with nontrivial nuclei, as we will show with the next example and in the following section. Given a residuated lattice $\mathbf{A}$, let $\mathsf{V}(\mathbf{A})$ denote the variety generated by $\mathbf{A}$. Notice that $t(x)$ is a nucleus on $\mathsf{V}(\mathbf{A})$ if and only if $t^\mathbf{A}$ is a nucleus on $\mathbf{A}$.  Moreover, $\mathsf{V}(\mathbf{A}) \vDash t(\bot) \approx \bot$ if and only if $t^\mathbf{A}(\bot) = \bot$.
	\begin{example}\label{nontrivial_nucleus_example}
		Let $n > 2$ and $A_n = \lbrace \bot,a_1, \dots,a_n, \top \rbrace$ with the total order $\leq$ given by $\bot < a_1 < \cdots < a_n < \top$.	If we define 
		\begin{equation*}
			x \cdot y :=
			\begin{cases}
				a_1 & \text{if $x,y \neq \bot, \top$}, \\
				\min\lbrace x,y \rbrace & \text{else,}
			\end{cases}
		\end{equation*}
		and $x \imp y := \max \lbrace z \in A: x \cdot z \leq y \rbrace$, then $\mathbf{A}_n := \langle A_n, \min, \max, \cdot, \imp, \bot, \top \rangle$ is a totally ordered residuated lattice. It can be easily shown that
		\begin{equation*}
			x \imp y :=
			\begin{cases}
				\top & \text{if $x \leq y$}, \\
				y & \text{if $x > y = \bot$ or $\top = x > y$}, \\
				a_n &  \text{if $\top > x > y > \bot$}.
			\end{cases}
		\end{equation*}
		
		If we define $t(x) := (x \imp x^2) \imp x^2$, then
		\begin{equation*}
			t^{\mathbf{A}_n}(x) = 
			\begin{cases}
				(x \imp a_1) \imp a_1 & \text{if $x \neq \bot$}, \\
				\bot & \text{if $x = \bot$},
			\end{cases}
		\end{equation*} 
		and $t^{\mathbf{A}_n}$ is a nucleus on $\mathbf{A}_n$. Therefore, $t(x)$ is a nucleus on $\mathsf{V}(\mathbf{A}_n)$ and it is nontrivial, because $a_2 < a_n = t^\mathbf{A}(a_2) < \top = \neg\neg a_2$.
	\end{example}
	
	We have given many examples of varieties with trivial and nontrivial nuclei. Nevertheless, an important question still remains unanswered: 
	
	\vspace{0.1in}
	
	\noindent\textbf{Open Problem.} Are there nontrivial nuclei on $\mathcal{RL}$?
	
	\vspace{0.1in}
	
	Since we are in lack of an answer for the previous question, we will restrict ourselves to proper subvarieties of $\mathcal{RL}$. Specifically, we will focus on varieties of BL-algebras. In this case, a very detailed study is possible. This will be the main topic of the following section.
	
	\section{Nuclei on varieties of BL-algebras}\label{sec:nuclei_on_BL}
	
	In this section we are going to study nuclei on BL-algebras and varieties of BL-algebras. The main result of this section consists of a complete description of all nuclei on an arbitrary subvariety of $\mathcal{BL}$. As we will show later, there exist varieties of BL-algebras with nontrivial nuclei. Moreover, a syntactic description of these terms is possible.
	
	We start studying properties of an arbitrary nucleus on a single BL-chain. Let $\mathbf{A}$ be nontrivial BL-chain and $\gamma$ a nucleus on $\mathbf{A}$. By Theorem~\ref{BLchains_decomposition}, there exists a total order $\langle I, \leq \rangle$ with first element $0$ and some set of nontrivial totally ordered Wajsberg hoops $\lbrace \mathbf{A}_i: i \in I \rbrace$, such that $A_i \cap A_j = \lbrace \top \rbrace$ for all $i \neq j$, $\mathbf{A}_0$ has first element $\bot^\mathbf{A}$ and $\mathbf{A} = (\bigoplus_{i \in I} \mathbf{A}_i)^+$. Consider the sets
	\begin{equation*}
		I_b := \lbrace i \in I: \mathbf{A}_i \text{ has first element} \rbrace \quad \text{and} \quad I_u := \lbrace i \in I: \mathbf{A}_i \text{ is unbounded} \rbrace.
	\end{equation*}
	
	Clearly $I_b \cap I_u = \emptyset$ and $I_b \cup I_u = I$. We start with some basic properties.
	\begin{lemma}\label{BLchains_nuclei_local_values}
		Let $i \in I$ and $a \in A_i$. The following hold:
		\begin{enumerate}[label=\rm{(\roman*)}]
			\item $\gamma(a) \in A_i$ or $\gamma(a) \in A_j \backslash \lbrace \top \rbrace$ for some $j > i$,
			\item if $a = \bot_i$ and $\gamma(a) \in A_i$, then $\gamma(a) = \min A_i \cap \gamma(A)$,
			\item if $b \in A_i \cap \gamma(A)$ and $b \leq a$, then $\gamma(a) = a$,
			\item if $A_i$ has first element $\bot_i$, then $\gamma(a) = a \lor \gamma(\bot_i)$,
			\item if $\min A_i \cap \gamma(A)$ does not exist, then $\gamma(a) = a$,
			\item if $b := \min A_i \cap \gamma(A) < \top$, then $\gamma(a) = a \lor b$,
			\item if $\min A_i \cap \gamma(A) = \top$, then $\gamma(a) = \min A_j \cap \gamma(A)$ for some $j \geq i$.
		\end{enumerate}
	\end{lemma}
	\begin{proof} $ $
		We will restrict ourselves to the proofs of (iii) to (vi). If $b \in A_i \cap \gamma(A)$ and $b \leq a$ then $b \leq \gamma(a)$ and, by Proposition~\ref{nuclei_properties1}, $a \lor b \leq \gamma(a) \leq (a \imp b) \imp b$. Moreover, since $\mathbf{A}_i$ is a Wajsberg hoop, $(a \imp b) \imp b =  a \lor b$. This concludes the proof of (iii). The proof of item (iv) is similar. In this case $a \lor \gamma(\bot_i) \leq \gamma(a) \leq (a \imp \gamma(\bot_i)) \imp \gamma(\bot_i) = a \lor \gamma(\bot_i)$. Last equality holds even if $\gamma(\bot_i)$ is not a member of $A_i$: if $a = \top$ it is trivial and if $a \neq \top$, then it suffices to apply item (i) and the definition of the ordinal sum. We now continue with the proof of item (v). Suppose that $a \in A_i \backslash \lbrace \top \rbrace$. If $a \leq b$ for all $b \in A_i \cap \gamma(A)$, then $\gamma(a) \leq b$ for all $b \in A_i \cap \gamma(A)$. By hypothesis there exists $c \in (A_i \backslash \lbrace \top \rbrace) \cap \gamma(A)$. Therefore $a \leq \gamma(a) \leq c$ and because $A_i \backslash \lbrace \top \rbrace$ is convex we conclude that $\gamma(a) \in A_i$. Consequently, $\gamma(a)$ is the first element of $A_i \cap \gamma(A)$ and we reached a contradiction. We conclude that there exists some $b \in A_i \cap \gamma(A)$ such that $b \leq a$. Thus, $\gamma(a) = a$. Finally, let us prove item (vi). If $a \in A_i \backslash \lbrace \top \rbrace$, then $a \leq \gamma(a) \leq (a \imp b) \imp b = a \lor b < \top$. By convexity of $A_i \backslash \lbrace \top \rbrace$, $\gamma(a) \in A_i$ and therefore $a \lor b \leq \gamma(a)$. 
	\end{proof}
	\begin{remark}\label{BLchains_nuclei_bounded_component_remark}
		Notice that the value of $\gamma$ on a bounded component is completely determined by the value of $\gamma$ on the first element of that component. This result will be crucial in our study of nuclei on varieties of BL-algebras. It is similar to the characterization of nuclei on MV-algebras given in Corollary~\ref{MV_nuclei}. This similarity could have been expected because of the definition ordinal sums and the equivalence between MV-algebras and bounded Wajsberg hoops (Theorem~\ref{BL-basic_term_equivalence}).
	\end{remark}
	
	We now direct our focus to nuclei on varieties of BL-algebras. Given a subvariety of $\mathcal{BL}$ we will provide a complete description of all nuclei on it. Let $\mathcal{V}$ be an arbitrary subvariety of $\mathcal{BL}$ and $t(x)$ a nucleus on $\mathcal{V}$. In view of Proposition~\ref{nuclei_term_interval} we will assume that $\mathcal{V} \vDash t(\bot) \approx \bot$. Before starting we introduce the following BL-chains:
	\begin{equation*}
		\mathbf{S}_{1,1} := (\mathbf{S}_1 \oplus \mathbf{S}_1)^+ \quad \text{and} \quad \mathbf{S}_{1, \omega} := (\mathbf{S}_1 \oplus \mathbf{S}_\omega)^+.
	\end{equation*}	
	
	As a consequence of Proposition~\ref{BL_chains_generated_subalgebras}, observe that $\mathbf{S}_{1,1} \in \mathcal{V}$ if and only if there exists $\mathbf{A} \in \mathcal{V}_{\textnormal{fsi}}$ such that $I_b \backslash \lbrace 0 \rbrace \neq \emptyset$. Analogously, $\mathbf{S}_{1, \omega} \in \mathcal{V}$ if and only if there exists $\mathbf{A} \in \mathcal{V}_{\textnormal{fsi}}$ such that $I_u \neq \emptyset$.
	Both structures will play a key role in the description of all nuclei on $\mathcal{V}$. We start with this simple result, which states that only trivial nuclei exist on both $\mathsf{V}(\mathbf{S}_{1,1})$ and $\mathsf{V}(\mathbf{S}_{1, \omega})$. 
	\begin{lemma}\label{A_1_1_A_1_omega_nuclei}
		Let $t(x)$ be a term and $\mathbf{A} \in \lbrace \mathbf{S}_{1,1}, \mathbf{S}_{1, \omega} \rbrace$. If $t^{\mathbf{A}}$ is a nucleus on $\mathbf{A}$ and $t^{\mathbf{A}}(\bot) = \bot$, then $\mathbf{A} \vDash t(x) \approx s(x)$ for some $s(x) \in \lbrace x, \neg\neg x \rbrace$.
	\end{lemma}
	\begin{proof}
		First take $\mathbf{A} := \mathbf{S}_{1,1}$ and without loss of generality assume that $A = \lbrace \bot, \bot_1, \top \rbrace$, where $\bot_1$ is the first element of the second component. Notice that $\langle \bot_1 \rangle^\mathbf{A} = A$ and $\bot_1 \leq t^\mathbf{A}(\bot_1)$. Therefore, $t^\mathbf{A}(\bot_1) \in \lbrace \bot_1, \top \rbrace$. If $t^\mathbf{A}(\bot_1) = \bot_1$, then $\mathbf{A} \vDash t(x) \approx x$. If not, then $\mathbf{A} \vDash t(x) \approx \neg\neg x$.
		
		Now consider $\mathbf{A} := \mathbf{S}_{1, \omega}$ and suppose that $A = \lbrace \bot, \top \rbrace \cup \lbrace a^k: k \in \omega \rbrace$. In this case, $\langle a^k \rangle^\mathbf{A} = \lbrace \bot, \top  \rbrace \cup \lbrace a^{kn}: n \in \omega \rbrace$ and  $t^\mathbf{A}(a^k) \in \lbrace a^k, \top \rbrace$ for all $k \in \omega$. If $t^\mathbf{A}(a) = a$, then, as a consequence of previous observation and the monotonicity of $t^\mathbf{A}$, we conclude that $\mathbf{A} \vDash t(x) \approx x$. Now suppose that $t^\mathbf{A}(a) = \top$. If there exists some $m > 1$ such that $t^\mathbf{A}(a^m) = a^m$, then, using the equality \eqref{nuclei_equation}, we conclude that $a^{m-1} = a \imp a^m = t^\mathbf{A}(a) \imp a^m = a^m$. Clearly we reached a contradiction. Therefore $t^\mathbf{A}(a^m) = \top$ for all $m > 1$. In this case, $\mathbf{A} \vDash t(x) \approx \neg\neg x$. 
	\end{proof}
	
	If $\mathbf{A} \in \mathcal{V}_{\textnormal{fsi}}$, then $\mathbf{A} = (\bigoplus_{i \in I} \mathbf{A}_i)^+$, where $\langle I, \leq \rangle$ is some total order with first element $0$ and $\lbrace \mathbf{A}_i: i \in I \rbrace$ is some set of totally ordered Wajsberg hoops such that $A_i \cap A_j = \lbrace \top \rbrace$ for all $i \neq j$ and $\mathbf{A}_0$ has first element $\bot^\mathbf{A}$. 
	\begin{lemma}\label{A_1_1_A_1_omega_I_b_I_u}
		Let $t(x)$ be a nucleus on $\mathcal{V}$ such that $\mathcal{V} \vDash t(\bot) \approx \bot$.
		\begin{enumerate}[label=\rm{(\roman*)}]
			\item $\ensuremath{\left. t^\mathbf{A} \right|_{A_0}} = \ensuremath{\left. x^\mathbf{A} \right|_{A_0}} = \ensuremath{\left. (\neg\neg x)^\mathbf{A} \right|_{A_0}}$.
			\item If $\mathbf{S}_{1,1} \in \mathcal{V}$ and $s(x) \in \lbrace x, \neg\neg x \rbrace$, then $\mathbf{S}_{1,1} \vDash t(x) \approx s(x)$ if and only if $\ensuremath{\left. t^\mathbf{A} \right|_{A_i}} = \ensuremath{\left. s^\mathbf{A} \right|_{A_i}}$ for all $i \in I_b\backslash\lbrace 0 \rbrace$.
			\item If $\mathbf{S}_{1, \omega} \in \mathcal{V}$ and $s(x) \in \lbrace x, \neg\neg x  \rbrace$, then $\mathbf{S}_{1, \omega} \vDash t(x) \approx s(x)$ if and only if $\ensuremath{\left. t^\mathbf{A} \right|_{A_i}} = \ensuremath{\left. s^\mathbf{A} \right|_{A_i}}$ for all $i \in I_u$. 
		\end{enumerate} 
	\end{lemma}
	\begin{proof}
		Item (i) is an immediate consequence of equality \eqref{BLchains_negation} and Proposition~\ref{nuclei_term_interval}. Now take $s(x) \in \lbrace x, \neg\neg x  \rbrace$. Taking into consideration the hypothesis, Proposition~\ref{BL_chains_generated_subalgebras} and Lemma~\ref{BLchains_nuclei_local_values}, if $\mathbf{S}_{1,1} \in \mathcal{V}$, then
		\begin{equation*}
			\begin{aligned}
				\mathbf{S}_{1,1} \vDash t(x) \approx s(x) & \Leftrightarrow \langle \bot_i \rangle^\mathbf{A} \vDash t(x) \approx s(x) \\
				& \Leftrightarrow t^\mathbf{A}(\bot_i) = s^\mathbf{A}(\bot_i) \\
				& \Leftrightarrow \ensuremath{\left. t^\mathbf{A} \right|_{A_i}} = \ensuremath{\left. s^\mathbf{A} \right|_{A_i}}
			\end{aligned}
		\end{equation*}
		for all $i \in I_b\backslash\lbrace 0 \rbrace$. Furthermore, if $\mathbf{S}_{1, \omega} \in \mathcal{V}$, then
		\begin{equation*}
			\begin{aligned}
				\mathbf{S}_{1, \omega} \vDash t(x) \approx s(x) & \Leftrightarrow \langle a \rangle^\mathbf{A} \vDash t(x) \approx s(x) \ \forall \ a \in A_i \backslash \lbrace \top \rbrace \\
				& \Leftrightarrow t^\mathbf{A}(a) = s^\mathbf{A}(a) \ \forall \ a \in A_i \\ 
				& \Leftrightarrow \ensuremath{\left. t^\mathbf{A} \right|_{A_i}} = \ensuremath{\left. s^\mathbf{A} \right|_{A_i}}
			\end{aligned}
		\end{equation*}
		for all $i \in I_u$. 
	\end{proof}
	
	The previous lemma states that the value of $t(x)$ on a bounded component different from the first one (unbounded component) is completely determined by the value of $t(x)$ on $\mathbf{S}_{1,1}$ ($\mathbf{S}_{1, \omega}$), if such algebra is a member of the variety. Combining this with Lemma~\ref{A_1_1_A_1_omega_nuclei} will help us describe all nuclei on $\mathcal{V}$. We start with some simple cases.
	\begin{proposition} 
		Let $t(x)$ be a nucleus on $\mathcal{V}$ such that $\mathcal{V} \vDash t(\bot) \approx \bot$.
		\begin{enumerate}[label=\rm{(\roman*)}]
			\item If $\mathbf{S}_{1,1}$, $\mathbf{S}_{1, \omega} \notin \mathcal{V}$, then $\mathcal{V} \vDash t(x) \approx x$.
			\item If $\mathbf{S}_{1,1} \in \mathcal{V}$, $\mathbf{S}_{1, \omega} \notin \mathcal{V}$, then $\mathcal{V} \vDash t(x) \approx s(x)$ for some $s(x) \in \lbrace x, \neg\neg x  \rbrace$.
			\item If $\mathbf{S}_{1, \omega} \in \mathcal{V}$, $\mathbf{S}_{1,1} \notin \mathcal{V}$, then $\mathcal{V} \vDash t(x) \approx s(x)$ for some $s(x) \in \lbrace x, \neg\neg x  \rbrace$.
		\end{enumerate}
	\end{proposition}
	\begin{proof}
		First assume that $\mathbf{S}_{1,1}$, $\mathbf{S}_{1, \omega} \notin \mathcal{V}$. In this case $I = \lbrace 0 \rbrace$ for all $\mathbf{A} \in \mathcal{V}_{\textnormal{fsi}}$. Therefore, $\mathcal{V}_{\textnormal{fsi}}$ is contained in $\mathcal{MV}$. By Proposition~\ref{nuclei_involutive_varieties} we obtain the desired result. Now suppose that $\mathbf{S}_{1,1} \in \mathcal{V}$, $\mathbf{S}_{1, \omega} \notin \mathcal{V}$. By Lemma~\ref{A_1_1_A_1_omega_nuclei}, $\mathbf{S}_{1,1} \vDash t(x) \approx s(x)$ where $\lbrace x, \neg\neg x  \rbrace$. Moreover, $I = I_b$ for all $\mathbf{A} \in \mathcal{V}_{\textnormal{fsi}}$. By Lemma~\ref{A_1_1_A_1_omega_I_b_I_u}, if $\mathbf{S}_{1,1} \vDash t(x) \approx s(x)$, then $\mathbf{A} \vDash t(x) \approx s(x)$ for all $\mathbf{A} \in \mathcal{V}_{\textnormal{fsi}}$. We proved the desired result. If $\mathbf{S}_{1, \omega} \in \mathcal{V}$, $\mathbf{S}_{1,1} \notin \mathcal{V}$ then the proof is analogous to previous case. Observe that, if $\mathbf{S}_{1,1} \notin \mathcal{V}$, then $I = \lbrace 0 \rbrace \cup I_u$ for all $\mathbf{A} \in \mathcal{V}_{\textnormal{fsi}}$. 
	\end{proof}
	
	We already know that $\mathbf{S}_{1, \omega} \in \mathcal{V}$ if and only if $I_u \neq \emptyset$ for some $\mathbf{A} \in \mathcal{V}_{\textnormal{fsi}}$. This is due to unbounded components being cancellative. Nevertheless, $\mathbf{S}_{1, \omega}$ can also be obtained as a result of generating a subalgebra with an element of a bounded component. As an example, consider $\mathbf{A} = (\mathbf{S}_1 \oplus \mathbf{S}_1^\omega)^+$ and $a = (1,-n) \in S_1^\omega$ with $n \geq 1$. Then $a^{k+1} < a^k$ for all $k \in \omega$ and by Proposition~\ref{BL_chains_generated_subalgebras}, $\langle a \rangle^\mathbf{A} \cong \mathbf{S}_{1, \omega}$. In regards to this observation we introduce the following property for $\mathcal{V}$:
	\begin{equation*}
		\text{(T)} \colon \exists \ \mathbf{A} \in \mathcal{V}_{\textnormal{fsi}}, \ i \in I_b \backslash\lbrace 0 \rbrace, \ a \in A_i \mid a^{n+1} < a^n \text{ for all $n \geq 0$}.
	\end{equation*}
	
	Note that this property is equivalent to the following:
	\begin{equation*}
		\text{(T')} \colon \exists \ \mathbf{A} \in \mathcal{V}_{\textnormal{fsi}}, \ i \in I_b \backslash\lbrace 0 \rbrace, \ a \in A_i \mid \langle a \rangle^\mathbf{A} \cong \mathbf{S}_{1, \omega}
	\end{equation*}
	
	To get a better understanding about this property we define
	\begin{equation*}
		\mathbf{S}_{1,_1^\omega} := (\mathbf{S}_1 \oplus \mathbf{S}_1^\omega)^+.
	\end{equation*}  
	\begin{proposition}\label{propertyT_equivalence}
		$\mathcal{V}$ satisfies property $(T)$ if and only if $\mathbf{S}_{1,_1^\omega} \in \mathcal{V}$.
	\end{proposition}	
	\begin{proof} $ $
		Suppose that $\mathcal{V}$ satisfies property (T). Let $\mathbf{A} \in \mathcal{V}_{\textnormal{fsi}}$, $a \in A_i$ with $i \in I_b \backslash \lbrace 0 \rbrace$ be such that $a^{n+1} < a^n$ for all $n \in \omega$. Then, we are going to prove that $B := \lbrace a^n: n \in \omega \rbrace \cup \lbrace a^{n} \imp \bot_i: n \in \omega \rbrace$ is a subalgebra of $\mathbf{A}_i$. Take $m,n \in \omega$ and observe that $(b \imp \bot_i) \imp \bot_i = b$ for all $b \in A_i$. Then,
		\begin{equation*}
			a^0 = \top, \quad a^n \cdot a^m = a^{n+m},
		\end{equation*}
		\begin{equation*}
			(a^n \imp \bot_i) \imp a^m = \top, \quad a^n \imp (a^m \imp \bot_i) = a^{n+m} \imp \bot_i,
		\end{equation*}
		\begin{gather*}
			a^m \imp a^n =
			\begin{cases}
				a^{n-m} & \text{if $n > m$,} \\
				\top & \text{else},
			\end{cases}
			\quad 
			(a^n \imp \bot_i) \imp (a^m \imp \bot_i) = a^m \imp a^n,
		\end{gather*}
		\begin{equation*}
			\begin{aligned}
				a^m \cdot (a^n \imp \bot_i) & = ((a^m \cdot (a^n \imp \bot_i)) \imp \bot_i) \imp \bot_i \\
				& = (a^m \imp a^n) \imp \bot_i
			\end{aligned}
		\end{equation*}
		and
		\begin{equation*}
			\begin{aligned}
				(a^m \imp \bot_i) \cdot (a^n \imp \bot_i) & = (((a^m \imp \bot_i) \cdot (a^n \imp \bot_i)) \imp \bot_i) \imp \bot_i \\
				& = ((a^n \imp \bot_i) \imp a^m) \imp \bot_i \\
				& = \bot_i.
			\end{aligned} 
		\end{equation*}
		
		Let $\mathbf{B}$ be the hoop with universe $B$. The map defined by $(a^n \imp \bot_i) \mapsto (0,n)$ and $a^n \mapsto (1,-n)$ for all $n \in \omega$ establishes an isomorphism between $\mathbf{B}$ and $\mathbf{S}_1^\omega$. Therefore, considering Proposition~\ref{subhoops_of_ordinal_sum}, $(\mathbf{S}_1 \oplus \mathbf{S}_1^\omega)^+ \cong (\mathbf{S}_1 \oplus \mathbf{B})^+$ is a subalgebra of $\mathbf{A}$. The reciprocal is immediate. 
	\end{proof}
	
	The previous result states that $\mathbf{S}_{1,_1^\omega}$ works as a canonical witness for the property (T). We now continue our study of nuclei on $\mathcal{V}$. Starting now we will suppose that both $\mathbf{S}_{1,1}$ and $\mathbf{S}_{1, \omega}$ are members of $\mathcal{V}$.  We will consider two cases: the case where $\mathbf{S}_{1,_1^\omega}$ is a member of $\mathcal{V}$ and the case where it is not. Notice that both $\mathbf{S}_{1,1}$ and $\mathbf{S}_{1, \omega}$ are subalgebras of $\mathbf{S}_{1,_1^\omega}$. This simplifies our hypothesis if $\mathbf{S}_{1,_1^\omega} \in \mathcal{V}$. 
	\begin{proposition}
		If $\mathbf{S}_{1,_1^\omega} \in \mathcal{V}$, then $\mathcal{V} \vDash t(x) \approx s(x)$ for some $s(x) \in \lbrace x, \neg\neg x  \rbrace$.
	\end{proposition}
	\begin{proof}
		In this case $\mathbf{S}_{1,1} \vDash t(x) \approx s(x)$ implies $\mathbf{S}_{1, \omega} \vDash t(x) \approx s(x)$ for all $s(x) \in \lbrace x, \neg\neg x  \rbrace$. To prove this observe that  
		\begin{equation*}
			\begin{aligned}
				\mathbf{S}_{1,1} \vDash t(x) \approx s(x) & \Rightarrow \mathbf{S}_{1,_1^\omega} \vDash t(x) \approx s(x) \\ 
				& \Rightarrow \mathbf{S}_{1, \omega} \vDash t(x) \approx s(x).
			\end{aligned}
		\end{equation*}
		
		Therefore, by Lemma~\ref{A_1_1_A_1_omega_I_b_I_u}, the value of $t(x)$ on an arbitrary component different from the first one is completely determined by the value of $t(x)$ on $\mathbf{S}_{1,1}$. Specifically, if $\mathbf{S}_{1,1} \vDash t(x) \approx s(x)$, then $\mathbf{A} \vDash t(x) \approx s(x)$ for all $\mathbf{A} \in \mathcal{V}_{\textnormal{fsi}}$. 
	\end{proof}
	
	We now turn our focus to the case where $\mathbf{S}_{1,1}$, $\mathbf{S}_{1, \omega} \in \mathcal{V}$ and $\mathbf{S}_{1,_1^\omega} \notin \mathcal{V}$. In this case, by Proposition~\ref{propertyT_equivalence} and the definition of property (T), for every $\mathbf{A} \in \mathcal{V}_{\textnormal{fsi}}$, $i \in I_b \backslash \lbrace 0 \rbrace$ and $a \in A_i$ there exists some $n \geq 1$ such that $a^n = a^{n+1}$. As a consequence we define the following set:
	\begin{equation*}
		P_{\mathcal{V}} := \left\lbrace n \geq 1: \exists \ \mathbf{A} \in \mathcal{V}_{\textnormal{fsi}}, \ i \in I_b \backslash \lbrace 0 \rbrace, \ a \in A_i \mid a^{n+1} = a^n < a^{n-1} \right\rbrace. 
	\end{equation*}
	
	It can be easily shown that $P_{\mathcal{V}} \neq \emptyset$ if and only if $\mathbf{S}_{1,1} \in \mathcal{V}$. Moreover, the next proposition will show that if $\mathbf{S}_{1,_1^\omega} \notin \mathcal{V}$, then $P_{\mathcal{V}}$ is finite.
	\begin{proposition}\label{infiniteM_implies_propertyT}
		If $P_{\mathcal{V}}$ is infinite, then $\mathcal{V}$ satisfies property $(T)$.
	\end{proposition}
	\begin{proof}
		For each $n \in P_{\mathcal{V}}$ there exists a totally ordered Wajsberg hoop $\mathbf{A}_n$ with first element $\bot_n$ and $a_n \in A_n$ such that $a_n^{n+1} = a_n^{n} < a_n^{n-1}$ and $(\mathbf{S}_1 \oplus \mathbf{A}_n)^+ \in \mathcal{V}$. Let $U$ be a nonprincipal ultrafilter on $P_{\mathcal{V}}$. Then 
		\begin{equation*}
			\textstyle \mathbf{B} := (\prod_{n \in P_{\mathcal{V}}} (\mathbf{S}_1 \oplus \mathbf{A}_n)^+ ) \slash U \in \mathcal{V}_{\textnormal{fsi}}
		\end{equation*}
		Moreover, $\mathbf{B}$ is a two component chain: the condition $\lvert I \rvert \leq 2$ can be expressed with an identity (see \cite[Lemma~4.2]{AGLIANO2003}) and the element $c := \langle \bot_n: n \in P_{\mathcal{V}} \rangle \slash U$ satisfies $c^2 = c$ and $\bot < c < \top$, that is, the element $c$ necessarily corresponds to the first element of the second component.
		
		Without loss of generality suppose that $\mathbf{B} = (\mathbf{B}_0 \oplus \mathbf{B}_1)^+$, where $\mathbf{B}_0$ and $\mathbf{B}_1$ are totally ordered Wajsberg hoops with first element and $c$ is the first element of $\mathbf{B}_1 $. If we define $a := \langle a_n: n \in P_{\mathcal{V}} \rangle \slash U$, then $a \in B_1$ and $a^{m+1} < a^{m}$ for all $m \in \omega$. To prove this let $m \in \omega$. Then, it can be easily shown that $\lbrace k \in P_{\mathcal{V}}: a_k^{m+1} < a_k^m \rbrace = \lbrace k \in P_{\mathcal{V}}: k > m \rbrace$. Given that every nonprincipal ultrafilter over an infinite set contains all cofinite sets relative to that set, we conclude that $\lbrace k \in P_{\mathcal{V}}: a_k^{m+1} < a_k^m \rbrace \in U$. Since $m \in \omega$ is arbitrary, $\mathcal{V}$ satisfies property (T). 
	\end{proof}
	\begin{corollary}\label{finite_M}
		If $\mathcal{V}$ does not satisfy property $(T)$ and $P_{\mathcal{V}}$ is not empty, then $P_{\mathcal{V}}$ is finite and $a^m = a^{m+1}$ for all $\mathbf{A} \in \mathcal{V}_{\textnormal{fsi}}$, $i \in I_b \backslash \lbrace 0 \rbrace$, $a \in A_i$, where $m = \bigvee P_{\mathcal{V}}$.
	\end{corollary}
	
	We now continue our study of nuclei on $\mathcal{V}$. Since we are supposing that $\mathbf{S}_{1,1}, \mathbf{S}_{1, \omega} \in \mathcal{V}$ and $\mathbf{S}_{1,_1^\omega} \notin \mathcal{V}$, we can use the previous corollary. Using similar ideas as those used previously we consider four cases:
	\begin{enumerate}
		\item $\mathbf{S}_{1,1} \vDash t(x) \approx x$ and $\mathbf{S}_{1, \omega} \vDash t(x) \approx x$,
		\item $\mathbf{S}_{1,1} \vDash t(x) \approx \neg\neg x$ and $\mathbf{S}_{1, \omega} \vDash t(x) \approx \neg\neg x$,
		\item $\mathbf{S}_{1,1} \vDash t(x) \approx x$ and $\mathbf{S}_{1, \omega} \vDash t(x) \approx \neg\neg x$,
		\item $\mathbf{S}_{1,1} \vDash t(x) \approx \neg\neg x$ and $\mathbf{S}_{1, \omega} \vDash t(x) \approx x$.
	\end{enumerate}
	
	The first two cases are immediate: the first one implies $\mathcal{V} \vDash t(x) \approx x$ and the second one $\mathcal{V} \vDash t(x) \approx \neg\neg x$ (by Lemma~\ref{A_1_1_A_1_omega_I_b_I_u}). The last two cases give place to nontrivial nuclei on $\mathcal{V}$. If $\mathbf{S}_{1,1} \vDash t(x) \approx x$ and $\mathbf{S}_{1, \omega} \vDash t(x) \approx \neg\neg x$, then we are looking for a term such that
	\begin{enumerate}
		\item[$\bullet$] $\ensuremath{\left. t^\mathbf{A} \right|_{A_0}} = \ensuremath{\left. x^\mathbf{A} \right|_{A_0}}$,
		\item[$\bullet$] $\ensuremath{\left. t^\mathbf{A} \right|_{A_i}} = \ensuremath{\left. x^\mathbf{A} \right|_{A_i}}$ for all $i \in I_b \backslash \lbrace 0 \rbrace$ and
		\item[$\bullet$] $\ensuremath{\left. t^\mathbf{A} \right|_{A_i}} = \ensuremath{\left. (\neg\neg x)^\mathbf{A} \right|_{A_i}} = \ensuremath{\left. \top^\mathbf{A} \right|_{A_i}}$ for all $i \in I_u$
	\end{enumerate}
	for all $\mathbf{A} \in \mathcal{V}_{\textnormal{fsi}}$. Notice that the term $\neg\neg x$ distinguishes elements of $A_0$ from elements of $A_i$ with $i \in I\backslash\lbrace 0 \rbrace$. Also observe that, if $a \in A_i$ with $i \in I\backslash\lbrace 0 \rbrace$, then
	\begin{equation*}
		a^m \imp a^{m+1} =
		\begin{cases}
			\top & \text{if $a \in A_i$, $i \in I_b \backslash \lbrace 0 \rbrace$}, \\
			a & \text{if $a \in A_i$, $i \in I_u$},
		\end{cases} 
	\end{equation*}
	where $m = \bigvee P_{\mathcal{V}}$. Therefore, the term $x^m \imp x^{m+1}$ allows us to distinguish elements of components with first element (different from $\mathbf{A}_0$) from elements of unbounded components. If we define
	\begin{equation*}
		s_m(x) := \neg\neg x \land ((x^m \imp x^{m+1}) \imp x),
	\end{equation*}
	then it can be easily shown that $\mathbf{A} \vDash t(x) \approx s_m(x)$ for all $\mathbf{A} \in \mathcal{V}_{\textnormal{fsi}}$. Now consider the case where $\mathbf{S}_{1,1} \vDash t(x) \approx \neg\neg x$ and $\mathbf{S}_{1, \omega} \vDash t(x) \approx x$. Following similar ideas as before, if we define
	\begin{equation*}
		t_m(x) := \neg\neg x \land (x^m \imp x^{m+1}),
	\end{equation*}
	then $\mathbf{A} \vDash t(x) \approx t_m(x)$ for all $\mathbf{A} \in \mathcal{V}_{\textnormal{fsi}}$. We summarize all results with the following theorem.
	\begin{theorem}
		If $\mathbf{S}_{1,1}$, $\mathbf{S}_{1, \omega} \in \mathcal{V}$ and $\mathbf{S}_{1,_1^\omega} \notin \mathcal{V}$, then $\mathcal{V} \vDash t(x) \approx s(x)$ for some $s(x) \in \lbrace x,s_m(x),t_m(x), \neg\neg x \rbrace$ and $m = \bigvee P_{\mathcal{V}}$.
	\end{theorem}
	
	We want to do a more detailed study about the last case. Unlike previous ones, we now are in the presence of nontrivial nuclei. The last theorem lists all possible nuclei on $\mathcal{V}$ but does not guarantee that $s_m(x)$ or $t_m(x)$ is a nucleus on $\mathcal{V}$. We start with a characterization of $s_m(x)$ and $t_m(x)$ being nuclei on an arbitrary variety $\mathcal{V} \subseteq \mathcal{BL}$.
	\begin{theorem}\label{s_m_t_m_characterization}
		$s_m(x)$ is a nucleus on $\mathcal{V}$ if and only if, for all $\mathbf{A} \in \mathcal{V}_{\textnormal{fsi}}$,
		\begin{enumerate}
			\item[-] $i < j$ for all $i \in I_b \backslash \lbrace 0 \rbrace$, $j \in I_u$ and
			\item[-] $a^m = a^{m+1}$ for all $a \in A_i$, $i \in I_b\backslash \lbrace 0 \rbrace$.
		\end{enumerate}
		Analogously, $t_m(x)$ is a nucleus on $\mathcal{V}$ if and only if, for all $\mathbf{A} \in \mathcal{V}_{\textnormal{fsi}}$,
		\begin{enumerate}
			\item[-] $i < j$ for all $i \in I_u$, $j \in I_b \backslash \lbrace 0 \rbrace$ and
			\item[-] $a^m  = a^{m+1}$ for all $a \in A_i$, $i \in I_b\backslash \lbrace 0 \rbrace$.
		\end{enumerate}
	\end{theorem}
	\begin{proof}
		We will restrict ourselves to the proof for $s_m(x)$, since both proofs are very similar. Before starting notice that
		\begin{equation*}
			s_m^\mathbf{A}(a) = 
			\begin{cases}
				a & \text{if $a \in A_0$}, \\
				(a^m \imp a^{m+1}) \imp a & \text{if $a \in A_i$, $i \in I_b \backslash\lbrace 0 \rbrace$}, \\
				\top & \text{if $a \in A_i$, $i \in I_u$}.
			\end{cases}
		\end{equation*}
		for all $\mathbf{A} \in \mathcal{V}_{\textnormal{fsi}}$ and $a \in A$. Now suppose that $s_m(x)$ is a nucleus on $\mathcal{V}$ and let $\mathbf{A}$ be a member of $\mathcal{V}_{\textnormal{fsi}}$. If $i \in I_b$, then $s_m^\mathbf{A}(\bot_i) = \bot_i$ and by Lemma~\ref{BLchains_nuclei_local_values}, $\ensuremath{\left. s_m^\mathbf{A} \right|_{A_i}} = \ensuremath{\left. x^\mathbf{A} \right|_{A_i}}$. By monotonicity of $s_m^\mathbf{A}$ we conclude that $i < j$ for all $i \in I_b \backslash \lbrace 0 \rbrace$, $j \in I_u$. At last, if $a \in A_i$ with $i \in I_b \backslash \lbrace 0 \rbrace$, then
		\begin{equation*}
			\begin{aligned}
				(a^m \imp a^{m+1}) \imp a = a & \Rightarrow ((a^m \imp a^{m+1}) \imp a) \imp a = \top \\
				& \Rightarrow a \lor (a^m \imp a^{m+1}) = \top \\
				& \Rightarrow a^m \imp a^{m+1} = \top. 
			\end{aligned}
		\end{equation*}
		
		Now suppose that $i < j$ for all $i \in I_b \backslash \lbrace 0 \rbrace$, $j \in I_u$ and $a^m = a^{m+1}$ for all $\mathbf{A} \in \mathcal{V}_{\textnormal{fsi}}$, $a \in A_i$, $i \in I_b\backslash \lbrace 0 \rbrace$ . Second hypothesis implies that 
		\begin{equation*}
			s_m^\mathbf{A}(a) =
			\begin{cases}
				a & \text{if $a \in I_b$}, \\
				\top & \text{if $a \in I_u$}
			\end{cases}
		\end{equation*}
		for all $a \in A$ and $\mathbf{A} \in \mathcal{V}_{\textnormal{fsi}}$. Using the first hypothesis we come to the conclusion that $\mathbf{A} \vDash s_m(x) \imp s_m(y) \approx x \imp s_m(y)$ for all $\mathbf{A} \in \mathcal{V}_{\textnormal{fsi}}$. This ends the proof. 
	\end{proof}
	
	Before continuing we introduce the following definitions:
	\begin{equation*}
		\mathbf{S}_{1,1, \omega} := (\mathbf{S}_1 \oplus \mathbf{S}_1 \oplus \mathbf{S}_\omega)^+ \quad \text{and} \quad \mathbf{S}_{1, \omega,1} := (\mathbf{S}_1 \oplus \mathbf{S}_\omega \oplus \mathbf{S}_1)^+.
	\end{equation*}	
	
	By Proposition~\ref{BL_chains_generated_subalgebras}, $\mathbf{S}_{1,1, \omega} \in \mathcal{V}$ if and only if there exists $\mathbf{A} \in \mathcal{V}_{\textnormal{fsi}}$ and $i \in I_b \backslash \lbrace 0 \rbrace$, $j \in I_u$ such that $i < j$. Analogously, $\mathbf{S}_{1, \omega,1} \in \mathcal{V}$ if and only if there exists $\mathbf{A} \in \mathcal{V}_{\textnormal{fsi}}$ and $i \in I_u$, $j \in I_b \backslash \lbrace 0 \rbrace$ such that $i < j$. These remarks together with Corollary~\ref{finite_M} and Theorem~\ref{s_m_t_m_characterization} lead us to the next result.
	\begin{corollary}
		Let $\mathcal{V}$ be such that $\mathbf{S}_{1,1} \in \mathcal{V}$, $\mathbf{S}_{1,_1^\omega} \notin \mathcal{V}$. Consider $m = \bigvee P_{\mathcal{V}}$.
		\begin{enumerate}[label=\rm{(\roman*)}]
			\item $s_m(x)$ is a nucleus on $\mathcal{V}$ if and only if $\mathbf{S}_{1, \omega,1} \notin \mathcal{V}$. 
			\item $t_m(x)$ is a nucleus on $\mathcal{V}$ if and only if $\mathbf{S}_{1,1, \omega} \notin \mathcal{V}$.
		\end{enumerate}
	\end{corollary}										
	
	Observe that both $\mathbf{S}_{1,1}$ and $\mathbf{S}_{1, \omega}$ are subalgebras of $\mathbf{S}_{1, \omega,1}$ and $\mathbf{S}_{1,1, \omega}$. This fact will allow us to simplify some of the hypotheses. We now summarize all cases with a single theorem.
	\begin{theorem}\label{BL_nuclei}
		Let $\mathcal{V}$ be a subvariety of $\mathcal{BL}$ and $t(x)$ a nucleus on $\mathcal{V}$.
		\begin{enumerate}[label=\rm{(\roman*)}]
			\item If  \ - $\mathbf{S}_{1,1}$, $\mathbf{S}_{1, \omega} \notin \mathcal{V}$ or
			\begin{enumerate}
				\item[-] $\mathbf{S}_{1,1} \in \mathcal{V}$ and $\mathbf{S}_{1, \omega} \notin \mathcal{V}$ or
				\item[-] $\mathbf{S}_{1, \omega} \in \mathcal{V}$ and $\mathbf{S}_{1,1} \notin \mathcal{V}$ or
				\item[-] $\mathbf{S}_{1,_1^\omega} \in \mathcal{V}$ or
				\item[-] $\mathbf{S}_{1,1, \omega}$, $\mathbf{S}_{1, \omega,1} \in \mathcal{V}$ and $\mathbf{S}_{1,_1^\omega} \notin \mathcal{V}$, 
			\end{enumerate} 
			then $t(x)$ is a nucleus on $\mathcal{V}$ if and only if it is trivial.
			\item If $\mathbf{S}_{1,1, \omega} \in \mathcal{V}$ and $\mathbf{S}_{1,_1^\omega}$, 
			$\mathbf{S}_{1, \omega,1} \notin \mathcal{V}$, then $t(x)$ is a nontrivial nucleus on $\mathcal{V}$ if and only if $\mathcal{V} \vDash t(x) \approx s_m(x)$, where $m = \bigvee P_{\mathcal{V}}$.
			\item If $\mathbf{S}_{1, \omega,1} \in \mathcal{V}$ and $\mathbf{S}_{1,_1^\omega}$, $\mathbf{S}_{1,1, \omega} \notin \mathcal{V}$, then $t(x)$ is a nontrivial nucleus on $\mathcal{V}$ if and only if $\mathcal{V} \vDash t(x) \approx t_m(x)$, where $m = \bigvee P_{\mathcal{V}}$.
			\item If $\mathbf{S}_{1,1}$, $\mathbf{S}_{1, \omega} \in \mathcal{V}$ and $\mathbf{S}_{1,_1^\omega}$, $\mathbf{S}_{1, \omega,1}$, $\mathbf{S}_{1,1, \omega} \notin \mathcal{V}$, then $t(x)$ is a nontrivial nucleus on $\mathcal{V}$ if and only if $\mathcal{V} \vDash t(x) \approx s(x)$ for some $s(x) \in \lbrace s_m(x), t_m(x) \rbrace$, where $m = \bigvee P_{\mathcal{V}}$.
		\end{enumerate}	
	\end{theorem}
	
	To end this section let us make some remarks about the previous theorem. Before continuing we recall a well-known result about nuclei. A proof of the following theorem may be found in \cite{ZHAO2013}.
	\begin{theorem}\label{nuclei_glivenko}
		Let $\mathbf{A}$ be a residuated lattice and $\gamma$ a nucleus on $\mathbf{A}$. The following are equivalent:
		\begin{enumerate}[label=\rm{(\roman*)}]
			\item $\gamma(\gamma(a) \imp a) = \top$ for all $a \in A$;
			\item $\gamma$ is a homomorphism from $\mathbf{A}$ onto $\mathbf{A}_\gamma$, and in this case, $\mathbf{A} \slash D_\gamma(A) \cong \mathbf{A}_\gamma$.
		\end{enumerate}
	\end{theorem}
	
	The equivalent conditions presented in the previous theorem are known as the \emph{$\gamma$-relative Glivenko property} \cite{ZHAO2013}. This result shows that, in some cases, if $\mathbf{A} \in \mathcal{V}$ and $\gamma$ is a nucleus on $\mathbf{A}$, then $\mathbf{A}_\gamma \in \mathcal{V}$. This is not generally true (see \cite[Corollary~3.9]{CASTAÑO2015}). 
	
	Let $\mathcal{V}$ be an arbitrary subvariety of $\mathcal{BL}$ and $t(x)$ a nucleus on $\mathcal{V}$. Given $\mathbf{A} \in \mathcal{V}$, let $\mathbf{A}_t$ and $D_t(\mathbf{A})$ denote the nuclear image and the set of dense elements of $\mathbf{A}$ relative to the nucleus $t^\mathbf{A}$, respectively. If $\mathbf{A} \in \mathcal{V}_{\textnormal{fsi}}$, then
	\begin{equation*}
		\mathbf{A}_t =
		\begin{cases}
			\mathbf{1} & \text{if $t(x) = \top$}, \\
			\mathbf{A} & \text{if $t(x) = x$}, \\
			\mathbf{A}_0^+ & \text{if $t(x) = \neg\neg x$}, \\
			(\bigoplus_{i \in I_b} \mathbf{A}_i)^+ & \text{if $t(x) = s_m(x)$, $m \geq 1$}, \\
			(\bigoplus_{i \in \lbrace 0 \rbrace \cup I_u} \mathbf{A}_i)^+ & \text{if $t(x) = t_m(x)$, $m \geq 1$},
		\end{cases}
	\end{equation*}
	and 
	\begin{equation*}
		D_t(\mathbf{A}) =
		\begin{cases}
			A & \text{if $t(x) = \top$}, \\
			\lbrace \top \rbrace & \text{if $t(x) = x$}, \\
			\lbrack \bigcup_{i \in I \backslash \lbrace 0 \rbrace} A_i \rbrack \cup \lbrace \top \rbrace & \text{if $t(x) = \neg\neg x$}, \\
			\lbrack \bigcup_{i \in I_u} A_i \rbrack \cup \lbrace \top \rbrace & \text{if $t(x) = s_m(x)$, $m \geq 1$}, \\
			\lbrack \bigcup_{i \in I_b \backslash \lbrace 0 \rbrace} A_i \rbrack \cup \lbrace \top \rbrace& \text{if $t(x) = t_m(x)$, $m \geq 1$}.
		\end{cases}
	\end{equation*}
	
	Notice that, as an immediate consequence of Lemma~\ref{A_1_1_A_1_omega_nuclei} and Lemma~\ref{A_1_1_A_1_omega_I_b_I_u},
	\begin{equation*}
		\ensuremath{\left. t^\mathbf{A} \right|_{A_i}} = \ensuremath{\left. x^\mathbf{A} \right|_{A_i}} \quad \text{or} \quad \ensuremath{\left. t^\mathbf{A} \right|_{A_i}} = \ensuremath{\left. \top^\mathbf{A} \right|_{A_i}}
	\end{equation*} 
	for all $i \in I$. That being the case, it can be easily shown that $\mathbf{A}$ satisfies the $t^\mathbf{A}$-relative Glivenko property, that is, the identity $t(t(x) \imp x) \approx \top$. Since $\mathbf{A}$ is an arbitrary member of $\mathcal{V}_{\textnormal{fsi}}$ we conclude that
	\begin{equation}\label{nuclei_glivenko_identity}
		\mathcal{V} \vDash t(t(x) \imp x) \approx \top.
	\end{equation}
	
	Hence, $t^\mathbf{A}$ is an homomorphism from $\mathbf{A}$ onto $\mathbf{A}_t$ and the quotient $\mathbf{A} \slash D_t(\mathbf{A})$ is isomorphic to $\mathbf{A}_t$ for all $\mathbf{A}$ in $\mathcal{V}$. 
	
	Now suppose that $\mathcal{V} \vDash t(\bot) \approx \bot$ and let $\mathbf{A} \in \mathcal{V}_{\textnormal{fsi}}$. Notice that, in this case, $\mathbf{A}_t$ is a subalgebra of $\mathbf{A}$. That is equivalent to saying that $\mathbf{A}$ satisfies the identities $t(\bot) \approx \bot$, $t(t(x) \cdot t(y)) \approx t(x) \cdot t(y)$, $t(t(x) \imp t(y)) \approx t(x) \imp t(y)$ and $t(\top) \approx \top$.	Since $\mathbf{A} \in \mathcal{V}_{\textnormal{fsi}}$ is arbitrary, $\mathcal{V}$ also satisfies the same identities. Therefore, $\mathbf{A}_t$ is a subalgebra of $\mathbf{A}$ for all $\mathbf{A}$ in $\mathcal{V}$. If $\mathcal{V} \vDash t(\bot) \approx \bot$, then $\mathcal{V} \vDash (x \imp t(x)) \land (t(x) \imp \neg\neg x) \approx \top$ (see Lemma~\ref{nuclei_term_interval}). Therefore, for all $\mathbf{A} \in \mathcal{V}$ and $a \in A$, 
	\begin{equation*}
		\begin{aligned}
			a & = a \land t(a) \\
			& = t(a) \cdot (t(a) \imp a) \\
			& = (t(a) \land \neg\neg a) \cdot (t(a) \imp a) \\
			& = \neg\neg a \cdot (\neg\neg a \imp t(a)) \cdot (t(a) \imp a).
		\end{aligned}
	\end{equation*}
	
	Moreover, $\neg\neg a \in A_{\neg\neg x}$, $\neg\neg a \imp t(a) \in D_{\neg\neg x}(A) \cap A_t$ and $t(a) \imp a \in D_t(A)$. We have shown that
	\begin{equation}\label{BL_nuclei_decomposition}
		\mathcal{V} \vDash x \approx \neg\neg x \cdot (\neg\neg x \imp t(x)) \cdot (t(x) \imp x)
	\end{equation}
	
	This generalizes the decomposition given in \cite{CIGNOLI2003}: 
	\begin{equation*}
		\mathcal{V} \vDash x \approx \neg\neg x \cdot (\neg\neg x \imp x).
	\end{equation*} 
	
	The identity \eqref{BL_nuclei_decomposition} is interesting when we consider nontrivial nuclei, that is, when $t(x) = s_m(x)$ or $t(x) = t_m(x)$ for some $m \in \omega$, $m \geq 1$. Starting now, suppose that $\mathbf{S}_{1,1} \in \mathcal{V}$ and $\mathbf{S}_{1,_1^\omega} \notin \mathcal{V}$. Taking into consideration Corollary~\ref{finite_M}, if $m = \bigvee P_{\mathcal{V}}$, then
	\begin{equation*}
		s_m(a) = 
		\begin{cases}
			a & \text{if $a \in A_0$;} \\
			a & \text{if $a \in A_i$, $i \in I_b \backslash \lbrace 0 \rbrace$;} \\
			\top & \text{if $a \in A_i$, $i \in I_u$;}
		\end{cases}
	\end{equation*}
	and
	\begin{equation*}
		t_m(a) = 
		\begin{cases}
			a & \text{if $a \in A_0$;} \\
			a & \text{if $a \in A_i$, $i \in I_u$;} \\
			\top & \text{if $a \in A_i$, $i \in I_b \backslash \lbrace 0 \rbrace$;} \\
		\end{cases}
	\end{equation*}
	for all $\mathbf{A} \in \mathcal{V}_{\textnormal{fsi}}$. As a consequence, it can be easily shown that
	\begin{equation*}
		\mathcal{V} \vDash s_m(x) \approx \neg\neg x \land (t_m(x) \imp x) \quad \text{and} \quad \mathcal{V} \vDash t_m(x) \approx \neg\neg x \land (s_m(x) \imp x).
	\end{equation*}
	
	Moreover,
	\begin{equation*}
		\mathcal{V} \vDash s_m(x) \land t_m(x) \approx x \quad \text{and} \quad \mathcal{V} \vDash s_m(x) \lor t_m(x) \approx \neg\neg x.
	\end{equation*}
	
	A well-known result states that, if $\mathbf{B}$ is a totally ordered Wajsberg hoop that satisfies the identity $x^m \approx x^{m+1}$ for some $m \in \omega$, then $\mathbf{B} \cong \mathbf{S}_n$ for some $n \leq m$ (see \cite{CORNISH1980}). Therefore, if $\mathbf{A} \in \mathcal{V}_{\textnormal{fsi}}$, then $\mathbf{A} \cong (\bigoplus_{i \in I} \mathbf{A}_i)^+$ for some totally ordered set $\langle I, \leq \rangle$ with first element $0$ and a set of totally ordered Wajsberg hoops $\lbrace \mathbf{A}_i: i \in I \rbrace$ such that $\mathbf{A}_0$ has first element $\bot^\mathbf{A}$, and, for all $i \in I \backslash \lbrace 0 \rbrace$, $\mathbf{A}_i = \mathbf{S}_{n}$ for some $n \leq m$ or $\mathbf{A}_i$ is unbounded. Notice that all properties mentioned until now hold whether $s_m(x)$ (or $t_m(x)$) is a nucleus on $\mathcal{V}$ or not. If it is the case that $s_m(x)$ is a nucleus on $\mathcal{V}$, then $i < j$ for all $i \in I_b \backslash \lbrace 0 \rbrace$, $j \in I_u$. Analogously, if $t_m(x)$ is a nucleus on $\mathcal{V}$, then $i < j$ for all $i \in I_u$, $j \in I_b \backslash \lbrace 0 \rbrace$. If both $s_m(x)$ and $t_m(x)$ are nuclei on $\mathcal{V}$, then $I = \lbrace 0 \rbrace \cup I_u$ or $I = I_b$.
	
	We end this section with a simple application of the identity \eqref{nuclei_glivenko_identity}. The variety $\mathcal{BL}$ is the algebraic counterpart of the \emph{Basic Fuzzy Logic}, BL for short. Moreover, axiomatic extensions of BL are in correspondence with subvarieties of $\mathcal{BL}$. By a propositional formula we understand a formula built from the variables $x_1, \dots,x_n$ with the connectives $\&$ (conjunction), $\Rightarrow$ (implication) and the constant $0$. Given a BL-algebra $\mathbf{A}$ and $a_1, \dots,a_n \in A$, $\varphi^\mathbf{A}(a_1, \dots,a_n)$ is the element of $A$ obtained by replacing in $\varphi$ the connectives $\&$, $\Rightarrow$ and $0$ by the operations $\cdot, \imp$ and $\bot$, respectively, and each $x_i$ by $a_i$ for $i = 1, \dots,n$. If $\mathcal{V}$ is a subvariety of $\mathcal{BL}$ and $\text{BL}_\mathcal{V}$ is its corresponding extension, then a propositional formula $\varphi(x_1, \dots,x_n)$ is derivable in $\text{BL}_\mathcal{V}$ if and only if, for each $\mathbf{A} \in \mathcal{V}$ and $a_1, \dots,a_n \in A$, $\varphi^\mathbf{A}(a_1, \dots,a_n) = \top$. For more about BL-algebras and the logic BL we refer to \cite{HAJEK1998}. 
	
	Consider $\mathcal{V}$ a subvariety of $\mathcal{BL}$, $t(x)$ a nucleus on $\mathcal{V}$ and $\varphi(x_1, \dots,x_n)$ a propositional formula. The following result is a direct consequence of the identities $t(t(x)) \approx t(x)$ and $t(t(x) \imp x) \approx \top$ (see Theorem~\ref{nuclei_glivenko}).
	\begin{proposition}\label{nucleus_homomorphism}
		If $\mathbf{A} \in \mathcal{V}$ and $a_1, \dots,a_n \in A$, then
		\begin{equation*}
			(t \circ \varphi)^\mathbf{A}(a_1, \dots,a_n) = t^\mathbf{A}(\varphi^\mathbf{A}(a_1, \dots,a_n)) = t^\mathbf{A}(\varphi^\mathbf{A}(t^\mathbf{A}(a_1), \dots,t^\mathbf{A}(a_n)))
		\end{equation*}
	\end{proposition}
	
	Let $\text{BL}_\mathcal{V}$ be the extension (of BL) corresponding to $\mathcal{V}$ and let $\text{BL}_\mathcal{U}$ be the extension corresponding to the variety $\mathcal{U}$ defined, relative to $\mathcal{V}$, by the identity $t(x) \approx x$. Notice that, as another consequence of the Glivenko property, if $\mathbf{A} \in \mathcal{V}$, then $\mathbf{A}_t \in \mathcal{U}$. The following result is a generalization of \cite[Theorem 2.1]{CIGNOLI2003}.
	\begin{theorem}\label{nucleus_generalized_Glivenko}
		Let $\varphi$ be a propositional formula. Then $\varphi$ is derivable in $\text{BL}_\mathcal{U}$ if and only if $t \circ \varphi$ is derivable in $\text{BL}_\mathcal{V}$.
	\end{theorem}
	\begin{proof}
		Suppose first that $\varphi$ is derivable in $\text{BL}_\mathcal{U}$. If $\mathbf{A} \in \mathcal{V}$ and $(a_1, \dots,a_n) \in A^n$, then $\mathbf{A}_t \in \mathcal{U}$ and $(t^\mathbf{A}(a_1), \dots,t^\mathbf{A}(a_n)) \in A_t^n$. Using the hypothesis be obtain that $\varphi^\mathbf{A}(t^\mathbf{A}(a_1), \dots,t^\mathbf{A}(a_n)) = \top$. Particularly, 
		\begin{equation*}
			t^\mathbf{A}(\varphi^\mathbf{A}(t^\mathbf{A}(a_1), \dots,t^\mathbf{A}(a_n))) = \top.
		\end{equation*}
		
		Therefore, $(t \circ \varphi)^\mathbf{A}(a_1, \dots,a_n) = \top$. This shows that $t \circ \varphi$ is derivable in $\text{BL}_\mathcal{V}$. Suppose now that $t \circ \varphi$ is derivable in  $\text{BL}_\mathcal{V}$ and let $\mathbf{A} \in \mathcal{U}$ and $(a_1, \dots,a_n) \in A^n$. Since $\mathcal{U} \subseteq \mathcal{V}$, $t^\mathbf{A}(\varphi^\mathbf{A}(a_1, \dots,a_n)) = \top$. By definition of $\mathcal{U}$ we conclude that $\varphi^\mathbf{A}(a_1, \dots,a_n) = \top$ and $\varphi$ is derivable in $\text{BL}_{\mathcal{U}}$. 
	\end{proof}

	\newpage
	
	Sebasti\'an Buss
	
	\noindent Instituto de Matem\'atica (INMABB), Departamento de Matem\'atica, Universidad Nacional del Sur (UNS)-CONICET, Bah\'ia Blanca, Argentina.
	
	\noindent sbuss94@gmail.com
	
	\
	
	Diego Casta\~no
	
	\noindent Departamento de Matem\'atica, Universidad Nacional del Sur (UNS), Bah\'ia Blanca, Argentina. \\
	Instituto de Matem\'atica (INMABB), Departamento de Matem\'atica, Universidad Nacional del Sur (UNS)-CONICET, Bah\'ia Blanca, Argentina.
	
	\noindent diego.castano@uns.edu.ar
	
	\
	
	Jos\'e Patricio D\'iaz Varela
	
	\noindent Departamento de Matem\'atica, Universidad Nacional del Sur (UNS), Bah\'ia Blanca, Argentina. \\
	Instituto de Matem\'atica (INMABB), Departamento de Matem\'atica, Universidad Nacional del Sur (UNS)-CONICET, Bah\'ia Blanca, Argentina.
	
	\noindent usdiavar@criba.edu.ar
	
\end{document}